\documentclass[12pt]{article}
\usepackage{mathrsfs}
\usepackage{amsmath}
\usepackage{float}
\usepackage{cite}
\usepackage{esint}
\usepackage[all]{xy}
\usepackage{amssymb}
\usepackage{amsthm}
\usepackage{color}
\usepackage{graphicx}
\usepackage{hyperref}
\usepackage{bm}
\usepackage{indentfirst}
\usepackage{geometry}
\theoremstyle{plain}
\newtheorem{thm}{Theorem}[section]
\newtheorem{lem}[thm]{Lemma}
\newtheorem{prop}[thm]{Proposition}

\newtheorem{cor}[thm]{Corollary}
\theoremstyle{definition}
\newtheorem{defn}[thm]{Definition}
\theoremstyle{remark}
\newtheorem{remark}[thm]{Remark}

\newcommand{\integral}{\mathrm{d}\mathscr{H}^n}
\numberwithin{equation}{section}

\begin{document}
\title{Integral type Gauss-Green formula on non-collapsed RCD spaces and its applications}
\author{\it Zhangkai Huang \thanks{School of Mathematics, Sun Yat-sen University, China. Email: \url{huangzhk27@mail.sysu.edu.cn}}}
\date{\small\today}
\maketitle
\begin{abstract}
	We prove an integral type Gauss–Green formula on non-collapsed  RCD spaces using the strong locality of the Laplacian and an eigenfunction approximation method. As applications, we generalize Colding's monotonicity formulas and prove an asymptotic formula linking the mean curvature of a hypersurface
	at a given point to the volume of small balls centered at that point.
\end{abstract}
\tableofcontents
\section{Introduction}

In the first decade of this century, Sturm \cite{St06a,St06b} and Lott-Villani \cite{LV09} independently introduced the notion of CD$(K,N)$ metric measure spaces using the language of optimal transport. This notion provides a synthetic formulation of a lower Ricci curvature bound $K\in \mathbb{R}$ and an upper dimension bound $N\in [1,\infty]$. After that, by adding the Riemannian structure to CD$(K,N)$ spaces, the class of RCD$(K,N)$ spaces was introduced. The precise definition of RCD$(K,N)$ spaces (and the equivalent ones) can be found in \cite{AGS14a,AMS19,G13,G15,EKS15}. Important examples of RCD$(K,N)$ spaces include the Ricci limit spaces arising in
the Cheeger-Colding theory \cite{CC96,ChCo1,ChCo2,ChCo3} as well as finite dimensional Alexandrov spaces with curvature bounded from below (see \cite[Appendix A]{ZZ10} and \cite{P11}).

In this paper, we mainly focus on a special subclass of RCD$(K, n)$ spaces, namely non-collapsed RCD$(K,n)$ spaces, for which the reference measure coincides with the $n$-dimensional Hausdorff measure $\mathscr{H}^n$. This definition was first introduced by De Philippis-Gigli \cite{DG18} as a synthetic counterpart of volume non-collapsed Gromov-Hausdorff limit spaces of Riemannian manifolds with a fixed dimension and a lower Ricci curvature bound. Compared to general RCD$(K,N)$ spaces, non-collapsed RCD$(K,n)$ spaces not only force $n$ to be an integer, but also admit stronger regularity properties. A key ingredient in the proof of our main result, Theorem \ref{stronglocality} below, is the following property that provides a connection between the Hessian and the Laplacian.
\begin{prop}[{\cite[Proposition 3.2]{H18}}]\label{prop1.1}
	Let $(X,\mathsf{d},\mathscr{H}^n)$ be a non-collapsed $\mathrm{RCD}(K,n)$ space. Then for any $f\in D(\Delta)$, the trace of the Hessian of $f$ coincides with the Laplacian of $f$, namely
	\[
	\langle\mathrm{Hess}\, f,\mathrm{g}\rangle=\Delta f\ \mathscr{H}^n\text{-a.e.}
	\]
	Here $\mathrm{g}$ is the canonical Riemannian metric on $({X},\mathsf{d},\mathscr{H}^n)$ defined by Gigli-Pasqualetto in \cite{GP22}.
\end{prop}

\begin{thm}[Strong locality of Laplacian]\label{stronglocality}
 Let $(X,\mathsf{d},\mathscr{H}^n)$ be a non-collapsed $\mathrm{RCD}(K,n)$ space. Then for any locally Lipschitz function $g\in D(\Delta)$ it holds $\mathrm{Hess}\, g=0$  $\mathscr{H}^n$-a.e. on $\{|\nabla g|=0\}$. In particular, $\Delta g=0$ $\mathscr{H}^n$-a.e. on $\{|\nabla g|=0\}$.
\end{thm}

A direct application of Theorem \ref{stronglocality} is the following property of the gradient of the heat kernel.
\begin{prop}[Non-degeneration of the heat kernel gradient]
	Suppose $(X,\mathsf{d},\mathscr{H}^n)$ is a non-collapsed $\mathrm{RCD}(K,n)$ space and $U$, $V$ are bounded open subsets of $X$ such that $U\Subset V$. Denote by $\rho$, $\rho^U$ be the heat kernel of $X$ and the Dirichlet heat kernel of $U$, respectively. Then for any fixed $y\in X,z\in U$ and $t>0$, we have 
	\[
	|\nabla_x \rho(x,y,t)|\neq 0\ \mathscr{H}^n\text{-a.e. }x\in X\ \text{and} \ 	|\nabla_x \rho^U(x,z,t)|\neq 0\ \mathscr{H}^n\text{-a.e. } x\in U. 
	\]

\end{prop}
As a further application of Theorem \ref{stronglocality}, we obtain the integral type Gauss-Green formula. In the smooth setting this is trivial by Sard’s theorem, and hence also holds for Ricci limit spaces. In our non-collapsed RCD setting, however, the proof relies crucially on the strong locality of the Laplacian. By working with the integrated form (\ref{1.1}), we bypass the vanishing set $\{|\nabla f|=0\}$. The differentiated form is then obtained by differentiating with respect to $t$.
\begin{thm}[Integral type Gauss-Green formula]\label{GGthm}
Let $f$ be a Lipschitz function on a non-collapsed $\mathrm{RCD}(K,n)$ space $(X,\mathsf{d},\mathscr{H}^n)$ and $\Omega_t:=\{x\in X:f(x)\leqslant t\}$ be the corresponding $t$-sublevel set. Then for any Sobolev function $h$, any compactly supported test function $g$ on $X$ and any continuous function $\tau$ on $\mathbb{R}$ it holds that
	\begin{equation}\label{1.1}
	\int_{\Omega_t} \tau\circ f\,\langle\nabla f,h\nabla g\rangle\,\mathrm{d}\mathscr{H}^n=\int_{-\infty}^t \tau(s)\int_{ \Omega_s } \mathrm{div}(h\nabla g)\,\mathrm{d}\mathscr{H}^n\mathrm{d}s,\ \forall t\in\mathbb{R}.
	\end{equation}
	In particular, we have
	\[
	\frac{\partial }{\partial t}\int_{\Omega_t} \tau\circ f\,\langle\nabla f,h\nabla g\rangle\,\mathrm{d}\mathscr{H}^n=\tau(t)\int_{ \Omega_t } \mathrm{div}(h\nabla g)\,\mathrm{d}\mathscr{H}^n, \ \forall t\in\mathbb{R}.
	\]
\end{thm}
\begin{remark}
	If we interpret $ \langle\nabla f, h\nabla g\rangle/|\nabla f|$ as $0$ on $\{|\nabla f|=0\}$, then letting $\tau=1$ and combining co-area formula give
	\begin{equation}\label{jhfoiheaohfe}
	\int_{\partial \Omega_t} \langle \nabla f,h\nabla g\rangle/|\nabla f|\, \mathrm{dPer}(\Omega_t,\cdot)=\int_{\Omega_t }\mathrm{div}(h\nabla g)\,\integral,\ \forall t\in \mathbb{R}.
	\end{equation}
	Here $\mathrm{dPer}(\Omega_t,\cdot)$ refers to the perimeter measure supported on $\partial \Omega_t$. Although in the non-smooth setting the left hand side of (\ref{jhfoiheaohfe}) may not be well-defined, Theorem \ref{GGthm} provides a continuous representation of the following function: 
	\[
	t\mapsto 	\int_{\partial \Omega_t} \langle \nabla f,h\nabla g\rangle/|\nabla f|\, \mathrm{dPer}(\Omega_t,\cdot).
	\]
\end{remark}
An application of the integral type Gauss-Green formula is the generalization of \cite{C12} to the non-smooth setting.

Let $(X,\mathsf{d},\mathscr{H}^n)$ be a non-collapsed non-parabolic RCD$(0,n)$ space with $n\geqslant 3$, where by non-parabolic we mean $X$ satisfies
\begin{equation}\label{1.3}
\int_{1}^\infty \frac{1}{\mathscr{H}^n(B_t(x))}\,\mathrm{d}t<\infty,\ \forall x\in X.
\end{equation}
By using the heat kernel $\rho$ of $X$, we can define the associated Green function of $X$ as 
\[
G(x,y):=\int_0^\infty \rho(x,y,t)\,\mathrm{d}t.
\]

For every $x\in X$, the Bishop-Gromov inequality \cite{LV09,St06b} and \cite{DG18} ensures the existence of the following limits.
\[
\nu_x:=\lim_{r\to 0}\frac{\mathscr{H}^n(B_r(x))}{ r^n}\in (0,\omega_n],\ \ \mathcal{V}_x:=\lim_{r\to \infty}\frac{\mathscr{H}^n(B_r(x))}{ r^n}\in [0,\nu_x], 
\]
where $\omega_n$ denotes the volume of the unit ball in $\mathbb{R}^n$.

Define $\mathsf{b}_x:=(n(n-2)\nu_x \, G(x,\cdot))^{1/(2-n)}$ as the Green distance function and set
\[
V_x(t):=t^{-n}\int_{\{\mathsf{b}_x\leqslant t\}}{|\nabla \mathsf{b}|}^4\,\integral,
\]
\[
A_x(t):=t^{1-n}\int_{\{\mathsf{b}_x= t\}}{|\nabla \mathsf{b}|}^3\,\mathrm{dPer}(\{\mathsf{b}_x\leqslant t\},\cdot),\ W_{x}(t):=\int_1^t s^{-n}\int_{\{\mathsf{b}_x= s\}} ({|\nabla \mathsf{b}|}^3-|\nabla \mathsf{b}|)\,\mathrm{dPer}(\{\mathsf{b}_x\leqslant s\},\cdot)\mathrm{d}s.
\]
 These functions are not only locally Lipschitz on $(0,\infty)$ but also possess the following properties.
\begin{thm}[Monotonicity]\label{thm1.7}
	For all $x\in X$, the following holds.
	\begin{itemize}
		\item[$(1)$]The functions $A_x$, $V_x$ and $W_x$ are non-increasing on $(0,\infty)$.
		\item[$(2)$]The following functions are non-decreasing on $(0,\infty)$:
		\[
		A_x-2(n-1)V_x,\  t^{2-n}(A_x-n\,\nu_x),\ A_x-2(n-2)W_x.
		\]
	\end{itemize} 
	
\end{thm}

\begin{thm}[Asymptotic behavior]We have
	\[
	A_x(0):=\lim_{t\to 0} A_x=n\,\nu_x,\ V_x(0)=\lim_{t\to 0} V_x =\nu_x,\  \lim_{t\to \infty} A_x=n\,\nu_x \left(\frac{\nu_x}{\mathcal{V}_x}\right)^{\frac{2}{2-n}},\ 	\lim_{t\to \infty} V_x=\nu_x \left(\frac{\nu_x}{\mathcal{V}_x}\right)^{\frac{2}{2-n}}.
	\]
\end{thm}Note that a rigidity phenomenon occurs when monotonicity is not strict.
\begin{thm}[Rigidity]
	If one of the six functions in Theorem \ref{thm1.7} is not strictly monotone or $\inf W>-\infty$, then $X$ is a metric cone. If moreover the tangent space at $x$ is $\mathbb{R}^n$, then $X=\mathbb{R}^n$.
\end{thm}
Another application of the integral type Gauss-Green formula is the following asymptotic formula,  which generalizes the results in \cite{HT03,H25}. This formula connects the mean curvature of a hypersurface at a given point with the volume of small balls around that point in the non-smooth setting.
\begin{thm}\label{mcthm}
	Let $(X,\mathsf{d},\mathscr{H}^n)$ be a non-collapsed $\mathrm{RCD}(K,n)$ space. Assume $f$ is a locally Lipschitz function such that for any $p\in X$, there exists $\varepsilon>0$ such that $f\in D(\Delta,B_\varepsilon(p))$. Then at $\mathscr{H}^n$-a.e. point $x\in \{y\in X:|\nabla f|(y)\neq 0\}$, it holds that
	\begin{equation}\label{1.2}
		\begin{aligned}
			\lim_{t\rightarrow 0} \left(\frac{\mathscr{H}^n(B_t(x)\setminus \{y\in X:f(y)\leqslant f(x)\})}{\omega_n t^{n+1}}-\frac{1}{2t}\right)=\frac{\Delta f(x)}{|\nabla f|(x)}+\frac{\mathrm{Hess}f\,(\nabla f,\nabla f)(x)}{|\nabla f|^3(x)}.
		\end{aligned}
	\end{equation}
\end{thm}
\begin{remark}
	If $|\nabla f|\neq 0$ $\mathscr{H}^n$-a.e.~on $B_\varepsilon(x)$, then the right hand side of (\ref{1.2}) is exactly the divergence of $\nabla f/|\nabla f|$. Indeed, in the smooth setting, it coincides with the mean curvature of the level set $\{y\in X: f(y)=f(x)\}$ at  $x$. 
\end{remark}
\subsection{Organization of the paper}
In Section \ref{sec2}, we collect all notation, preliminary results and terminology concerning RCD$(K,N)$ spaces. Of particular importance is the strong locality of the Laplacian on non-collapsed RCD$(K,n)$ spaces. We also present a direct application, namely the almost everywhere non-zero property of the (Dirichlet) heat kernel.

In Section \ref{sec3}, using the strong locality of the Laplacian on non-collapsed RCD$(K,n)$ spaces, we prove an integral type Gauss–Green formula.  Throughout the argument, the Gaussian estimates of the heat kernel and the behavior of eigenfunctions play a crucial role.

Section \ref{sec5} applies the integral type Gauss–Green formula to show Colding’s monotonicity formulas. Through the integral type Gauss–Green formula, a deeper understanding of the regularity of $A_x$ is gained. Moreover, the monotonicity itself is proved using an integral type Bochner inequality, which we obtain in Section \ref{sec3} by combining the classical Bochner inequality with the integral type Gauss–Green formula. The rigidity statement, finally, follows from \cite{HP25}.

In Section \ref{sec4}, we once again call upon the integral type Gauss–Green formula, this time to derive an asymptotic formula for the level sets of nice functions. Following the strategy in \cite{H25}, we first construct normal coordinate charts near $\mathscr{H}^n$-a.e. point, understood in the sense that the limit of the average integral of the squared Hessian of the coordinate functions over a ball tends to zero as the radius shrinks to zero. Mapping back to Euclidean space then yields the desired conclusion.

\subsection{Acknowledgments}
The author thanks Yuanlin Peng for providing the proof of the Gauss-Green formula for level sets of gradient non-vanishing functions. He is also grateful to Huichun Zhang, Shouhei Honda for helpful discussions related to this paper.

\section{Preliminaries}\label{sec2}
In this paper we adopt the following conventions and notation.
\begin{itemize}
\item Denote by $C(K_1,\ldots,K_n)$ a positive constant depending only on $K_1,\ldots,K_n$.
\item Denote by $0_n$ the origin of $\mathbb{R}^n$ and by $\mathscr{L}^n$ the standard Lebesgue measure on $\mathbb{R}^n$.
\item By a metric measure space $(X,\mathsf{d},\mathfrak{m})$ we mean that $(X,\mathsf{d})$ is a complete and separable metric space, $\mathfrak{m}$ is a nonnegative Borel measure that is finite on bounded subsets of $X$, and $\operatorname{supp}(\mathfrak{m})=X$.
\item For a metric space $(X,\mathsf{d})$, we use the following notation:
\begin{itemize}
	\item $\operatorname{Lip}(X,\mathsf{d})$, $\operatorname{Lip}_{\mathrm{loc}}(X,\mathsf{d})$, $\operatorname{Lip}_{\mathrm{c}}(X,\mathsf{d})$, and $C(X)$ denote the spaces of Lipschitz, locally Lipschitz, compactly supported Lipschitz, and continuous functions, respectively. For any $f\in\operatorname{Lip}(X,\mathsf{d})$, the local Lipschitz constant at $x\in X$ is defined as
	\[
	\operatorname{lip} f(x) =
	\begin{cases}
		\limsup_{y\to x} \dfrac{|f(y)-f(x)|}{\mathsf{d}(y,x)}, & \text{if }x\text{ is not isolated},\\[1.2ex]
		0, & \text{otherwise}.
	\end{cases}
	\]
	\item  $B_R(x):=\{y\in X: \mathsf{d}(x,y)<R\}$ denotes the open ball of radius $R$ centered at $x$. In particular, $B_r(0_n):=\{x\in\mathbb{R}^n: |x|<r\}$.
	\item $\mathrm{diam}(X):=\sup_{x,y\in X}\mathsf{d}(x,y)$ denotes the diameter of $X$.

\end{itemize}
\end{itemize}

For a metric measure space $(X,\mathsf{d},\mathfrak{m})$, the Cheeger energy $\mathrm{Ch}:L^2(\mathfrak{m})\rightarrow [0,\infty]$ is a convex, lower semi-continuous functional defined as
\[
\mathrm{Ch}(f):=\inf_{\{f_i\}}\left\{\int_X \left(\mathrm{lip}\mathop{f_i}\right)^2\mathrm{d}\mathfrak{m}\right\},
\]
where the infimum is taken over all sequences $\{f_i\}\subset \mathrm{Lip}(X,\mathsf{d})\cap L^2(\mathfrak{m})$ converging to $f$ in $L^2(\mathfrak{m})$. The associated Sobolev space $H^{1,2}(X,\mathsf{d},\mathfrak{m})$ is then the collection of all $L^2$-functions with finite Cheeger energy. We denote by $H^{1,2}_{\text{loc}}(X,\mathsf{d},\mathfrak{m})$ the set of functions that lie in $H^{1,2}(U,\mathsf{d},\mathfrak{m})$ for every bounded open $U \subset X$; the space $L^p_{\text{loc}}(X,\mathfrak{m})$ is defined analogously. When no confusion arises, we simply write $L^p:=L^p(X,\mathfrak{m})$ and $H^{1,2}:=H^{1,2}(X,\mathsf{d},\mathfrak{m})$.

By Mazur's Lemma, for every $f\in H^{1,2}$, there exists a unique minimal relaxed slope $|\nabla f|\in L^2$ providing the canonical representation
\[
\mathrm{Ch}(f)=\int_X {|\nabla f|}^2\, \mathrm{d}\mathfrak{m}.
\]
{}{This object enjoys the locality property:} $|\nabla f|=|\nabla h|$ holds $\mathfrak{m}$-a.e. on $\{f=h\}$. 

A metric measure space is called infinitesimally Hilbertian if its Sobolev space $H^{1,2}$ carries a Hilbert space structure. As shown in \cite{AGS14a,G15}, under this assumption one can define for any $f,h\in H^{1,2}$ the following function in $L^1$:
\[
\langle\nabla f,\nabla h\rangle:=\lim_{\varepsilon\rightarrow 0}\frac{{|\nabla f+\varepsilon h|}^2-{|\nabla h|}^2}{2\varepsilon}.
\]

Let $U\subset X$ be an open set. A function $f \in H^{1,2}_{\text{loc}}(U, \mathsf{d},\mathfrak{m})$ is said to belong to the domain of the Laplacian on $U$, denoted  $f\in D(\Delta, U)$, if there exists an $h\in L^2(U,\mathfrak{m})$ such that the following integration by parts formula holds for every compactly supported $\psi \in H^{1,2}(U, \mathsf{d}, \mathfrak{m})$:

\[
\int_U \langle \nabla f, \nabla \psi\rangle \, \mathrm{d}\mathfrak{m} = -\int_U h\psi  \, \mathrm{d}\mathfrak{m}.
\]
This $h$ is unique if it exists and is denoted by $\Delta f$. In the case $U=X$, we simply write $f\in D(\Delta)$.

We are now in a position to introduce the definition of RCD$(K,N)$ spaces. See \cite{AGS15,AMS19,EKS15,G15} for details.
\begin{defn}
	For $K\in \mathbb{R}$, $N\in (1,\infty)$, a metric measure space $(X,\mathsf{d},\mathfrak{m})$ is said to be an RCD$(K,N)$ space if it satisfies the following conditions.
	\begin{enumerate}
		\item[$(1)$] {}{It is infinitesimally Hilbertian.}
		\item[$(2)$] There exists $x\in X$ and $C>1$ such that for any $r>0$ we have $\mathfrak{m}(B_r(x))\leqslant C\exp(Cr^2)$.
		\item[$(3)$] Any $f\in H^{1,2}$ satisfying $|\nabla f|\leqslant 1$ $\mathfrak{m}$-a.e. has a 1-Lipschitz representative.
		\item[$(4)$] For any $f\in D(\Delta)$ with $\Delta f\in H^{1,2}$ and any $\varphi \in \mathrm{Test}(X)$ with
		$ \varphi \geqslant 0$, we have
		\[
		\frac{1}{2}\int_X |\nabla f|^2 \Delta \varphi\mathop{\mathrm{d}\mathfrak{m}}\geqslant \int_X \varphi \left(\frac{(\Delta f)^2}{N}+\langle \nabla f,\nabla \Delta f\rangle+K|\nabla f|^2 \right)\mathrm{d}\mathfrak{m},
		\]
		where $\mathrm{Test}(X)$ is the class of test functions defined by
		\[
		\mathrm{Test}(X)=\mathrm{Test}(X,\mathsf{d},\mathfrak{m}):=\left\{\varphi\in \text{Lip}({X},\mathsf{d})\cap D(\Delta)\cap L^\infty:\Delta \varphi\in H^{1,2}\cap L^\infty\right\}.
		\]
	\end{enumerate}
	If in addition $\mathfrak{m}=\mathscr{H}^N$, then $(X,\mathsf{d},\mathfrak{m})$ is said to be a non-collapsed RCD$(K,N)$ space. 
\end{defn}

In the remainder of this paper, an RCD$(K,N)$ space always has $N \in (1,\infty)$. Let $(X, \mathsf{d}, \mathfrak{m})$ be such a space. By \cite[Theorem~1]{R12}, a local $(1,1)$-Poincar\'{e} inequality holds on $X$: 

\[
\fint_{B_r(x)} \left|f - \fint_{B_r(x)} f \, \mathrm{d}\mathfrak{m} \right|\, \mathrm{d}\mathfrak{m} \leqslant 4re^{|K|r^2} \fint_{B_{2r}(x)} |\nabla f| \, \mathrm{d}\mathfrak{m},\  \text{for any $f \in H^{1,2}$ and any ball $B_r(x)\subset X$}.
\]
Together with the following Bishop–Gromov volume growth inequality (see \cite[Theorem 5.31]{LV09}, \cite[Theorem 2.3]{St06b}), this implies that RCD$(K,N)$ spaces are PI spaces.
\begin{thm}\label{BGineq}
	For any $R>r>0$ $({}{\text{with}}\  R\leqslant \pi\sqrt{(N-1)/K}$ if $K>0)$, it holds 
	\[
	\dfrac{\mathfrak{m}\left(B_R(x)\right)}{\mathfrak{m}\left(B_r(x)\right)}\leqslant  \frac{V_{K,N}(R)}{ V_{K,N}(r)},
	\]
	where $V_{K,N} (r)$ denotes the volume of a ball of radius
	$r$ in the $N$-dimensional model space with Ricci curvature $K$ defined as
	\[ 
	{}{V_{K,N}(t):=\left\{
		\begin{array}{ll}
			\int_0^t\sin^{N-1}\left(s\sqrt{K/(N-1)}\right)\mathrm{d}s, &\text{if}\  K>0,\\
			t^{N},&\text{if}\  K=0,\\
			\int_0^t\sinh^{N-1}\left(s\sqrt{-K/(N-1)}\right)\mathrm{d}s, &\text{if}\  K<0.
		\end{array}
		\right.}
	\]  
\end{thm}

\subsection{Convergence of RCD spaces}

We omit the definition of pointed measured Gromov-Hausdorff (pmGH) convergence and instead recall the following precompactness result for RCD$(K,N)$ spaces (see  \cite{GMS15} for details).
\begin{thm}\label{GMS}
	Let $\{(X_i,\mathsf{d}_i,\mathfrak{m}_i,x_i)\}$ be a sequence of pointed $\mathrm{RCD}(K,N)$ spaces such that
	\[
	0<\liminf_{i\rightarrow \infty}\mathfrak{m}_i(B_1^{X_i}(x_i))\leqslant \limsup_{i\rightarrow \infty}\mathfrak{m}_i(B_1^{X_i}(x_i))<\infty.
	\]
	Then this sequence admits a subsequence $\{(X_{i(j)},\mathsf{d}_{i(j)},\mathfrak{m}_{i(j)},x_{i(j)})\}$ that $\mathrm{pmGH}$-converges to a pointed $\mathrm{RCD}(K,N)$ space $(X,\mathsf{d},\mathfrak{m},x)$.
\end{thm}

The \textit{regular sets} are then defined as follows.

\begin{defn}[Regular set]\label{regularset}
	Let $(X,\mathsf{d},\mathfrak{m})$ be an RCD$(K,N)$ space. The tangent space at $x\in X$, denoted by $\mathrm{Tan}(X,\mathsf{d},\mathfrak{m},x)$, is defined as
	\[
	\left\{(Y,\mathsf{d}_Y,\mathfrak{m}_Y,y):\exists r_i\downarrow 0 \text{, s.t. }\left(X,{r_i}^{-1}\mathsf{d},(\mathfrak{m}(B_{r_i}(x)))^{-1}\mathfrak{m},x\right)\xrightarrow{\mathrm{pmGH}}(Y,\mathsf{d}_Y,\mathfrak{m}_Y,y)\right\}.
	\]
	The {}{set of $k$-dimensional regular points} is then defined as
	\[
	\mathcal{R}_k:=\left\{x\in X:\mathrm{Tan}(X,\mathsf{d},\mathfrak{m},x)=\left\{\left(\mathbb{R}^k,\mathsf{d}_{\mathbb{R}^k},{\omega_k}^{-1}\mathscr{L}^k,0_k\right)\right\}\right\}.
	\]
\end{defn}
For the subsequent result concerning the existence of the \textit{essential dimension} in RCD spaces, we refer to \cite[Theorem 0.1]{BS20}.
\begin{thm}\label{BS}
	Let $(X,\mathsf{d},\mathfrak{m})$ be an $\mathrm{RCD}(K,N)$ space. There exists a unique integer $n\in [1,N]$ such that $\mathfrak{m}\left(X\setminus \mathcal{R}_n\right)=0$. This integer $n$ is referred to as the essential dimension of $(X,\mathsf{d},\mathfrak{m})$ and is denoted by $\mathrm{dim}_{\mathsf{d},\mathfrak{m}}(X)$.

\end{thm}

In the specific case of non-collapsed RCD$(K,n)$ spaces, the following result concerning  the essential dimension holds.
\begin{thm}[{\cite[Corollary 1.7]{DG18}}]\label{DG18cor1.7} Let $(X,\mathsf{d},\mathscr{H}^n)$ be a non-collapsed
	$\mathrm{RCD}(K, n)$ space. Then $\mathrm{dim}_{\mathsf{d},\mathscr{H}^n}(X)=n$. For every $x\in X$, the limit 
	$	\lim_{r\downarrow 0} r^{-n}{\mathscr{H}^n(B_r(x)})$ exists, is positive, and does not exceed $\omega_n$. Moreover, $x\in \mathcal{R}_n$ if and only if this limit equals $\omega_n$.
\end{thm}

In the remainder of this subsection, we consider a sequence of pointed RCD$(K,N)$ spaces $\{(X_i, \mathsf{d}_i, \mathfrak{m}_i, x_i)\}$ that converges in the pointed measured Gromov-Hausdorff (pmGH) sense to another pointed RCD$(K,N)$ space $(X, \mathsf{d}, \mathfrak{m}, x)$.

We assume that the reader is familiar with the definitions of $L^2$-weak and  $L^2$-strong convergence and $L^2_{\mathrm{loc}}$ (and their counterparts for Sobolev functions, namely $H^{1,2}$-weak and $H^{1,2}$-strong convergence, together with $H^{1,2}_{\mathrm{loc}}$-weak and $H^{1,2}_{\mathrm{loc}}$-strong convergence) on varying spaces. For references, see \cite{AST17,AH17,AH18,GMS15} and \cite[Definition 1.2]{H15}. We conclude this subsection by presenting some useful results related to this topic.
\begin{thm}[Arzel\`{a}-Ascoli theorem]\label{AAthm}
Assume a sequence $\{f_i\}$ satisfies $f_i\in C(X_i)$,
\[
\sup_i \sup_{y_i\in B_R(x_i)}|f_i(y_i)|<\infty,
\]
and the local equicontinuity condition: for any $\varepsilon,R\in (0,\infty)$, there exists $\delta\in (0,1)$ such that for any $i\in \mathbb{N}$ it holds that
\[
|f_i(y_i)-f_i(z_i)|<\varepsilon, \ \forall y_i,z_i\in B_R(x_i) \text{ such that } \mathsf{d}_i(y_i,z_i)<\delta.
\]
Then after passing to a subsequence, there exists $f\in C(X)$ such that 
\[
f_i(y_i)\rightarrow f(y) \text{\ whenever\ } X_i\ni y_i\rightarrow y\in X.
\]
\end{thm}

\begin{thm}[Compactness of local Sobolev functions \cite{AH18}]
Let $R>0$. Suppose $f_i\in H^{1,2}(B_R(x_i),\mathsf{d}_i,\mathfrak{m}_i)$ $(i\in\mathbb{N})$ satisfies $\sup_i \|f_i\|_{H^{1,2}}<\infty$. Then there exists $f\in H^{1,2}(B_R(x),\mathsf{d},\mathfrak{m})$ such that after passing to a subsequence, $\{f_i\}$ $L^2$-strongly converges to $f$ on $B_R(x)$ and
\[
\lim_{i\to \infty}\int_{B_R(x_i)}{|\nabla f_i|}^2\,\mathrm{d}\mathfrak{m}_i\geqslant \int_{B_R(x)}{|\nabla f|}^2\,\mathrm{d}\mathfrak{m}.
\]
\end{thm}

 \begin{thm}[Stability of Laplacian on balls  \cite{AH18}]\label{AH18}
 Let $f_i\in D(\Delta, B_R(x_i))$ $(i\in\mathbb{N})$ such that $\sup_i \|\Delta f_i\|_{L^2\left(B_R(x_i)\right)}<\infty$, and $\{f_i\}$ $L^2$-strongly converges to $f\in L^2(B_R(x),\mathfrak{m})$ on $B_R(x)$. Then for any $r\in (0,R)$ the following holds.
 \begin{enumerate}
 \item[$(1)$] $f|_{B_r(x)}\in D(\Delta,B_r(x))$.
 \item[$(2)$] $\{\Delta_i f_i\}$ $L^2$-weakly converges to $\Delta f$ on $B_r(x)$.
 \item[$(3)$] $\{f_i\}$ $H^{1,2}$-strongly converges to $f$ on $B_r(x)$.
 \end{enumerate}
 \end{thm}
 \begin{prop}[Harmonic approximation \cite{AH18}]\label{AH182}
Let  $f \in H^{1,2}(B_R(x), \mathsf{d}, \mathfrak{m})$ be a harmonic function $($i.e. $ f \in D(\Delta, B_R(x))$ with $ \Delta f = 0 )$. Then for any $ 0 < r < R $, there exist harmonic functions $ f_i \in H^{1,2}(B_r(x_i), \mathsf{d}_i, \mathfrak{m}_i)$ for each $i$ such that $\{f_i\}$ $H^{1,2} $-strongly converges to $f$ on $ B_r(x)$.
\end{prop}

\subsection{Calculus in RCD spaces}
{This subsection presents key results on calculus in} RCD$(K,N)$ spaces.  Let $(X,\mathsf{d},\mathfrak{m})$ be an RCD$(K,N)$ space. We omit definitions of the $L^p$-tangent module, $L^p$-cotangent module and $L^p$-tensor fields of type (0,2) on $X$ for $p\in [1,\infty]$ (denoted by $L^p(TX)$, $L^p(T^\ast X)$ and $L^p((T^\ast)^{\otimes 2}X)$, respectively), as well as the definition of the pointwise {Hilbert-Schmidt norm $|\cdot|_{\mathsf{HS}}$} (written as $|\cdot|$ for simplicity) for $L^p$-tensor fields. Details can be found in \cite{G18b}.

\begin{thm}[Exterior derivative]
	The linear operator $d:H^{1,2}\rightarrow L^2(T^\ast X)$, called the exterior derivative, satisfies $|d\, f|=|\nabla f|$ $\mathfrak{m}$-a.e. for any $f\in H^{1,2}$. Moreover, the set $\{d\, f:f\in H^{1,2}\}$ is dense in $L^2(T^\ast X)$.
\end{thm}

\begin{thm}[The canonical Riemannian metric \cite{GP22, AHPT21}]\label{111thm2.21}
	There exists a unique $\mathrm{g}\in L^\infty\left((T^\ast)^{\otimes 2}X\right)$ with  $\left|\mathrm{g}\right| =\sqrt{\mathrm{dim}_{\mathsf{d},\mathfrak{m}}({X})}$ $\mathfrak{m}$-a.e., called the canonical Riemannian metric of $(X,\mathsf{d},\mathfrak{m})$, such that 
	\[
	\mathrm{g}\left(\nabla f_1,\nabla f_2\right)=\left\langle \nabla f_1,\nabla f_2\right\rangle\ \ \text{$\mathfrak{m}$-a.e.}, \ \forall f_1,f_2\in H^{1,2}.
	\]
	
\end{thm}
\begin{thm}[The Hessian]
	For any $f\in \mathrm{Test}(X)$, there exists a unique $T\in L^2\left((T^\ast)^{\otimes 2}X\right)$, called the Hessian of $f$, denoted by $ \mathop{\mathrm{Hess}}f$, such that for all $f_i\in \mathrm{Test}(X)$ $(i=1,2)$,
	\begin{equation}\label{equation2.1}
		{}{2T(\nabla f_1,\nabla f_2)= \langle \nabla f_1,\nabla\langle \nabla f_2,\nabla f\rangle\rangle +\langle \nabla f_2,\nabla\langle \nabla f_1,\nabla f\rangle\rangle-\langle \nabla f,\nabla\langle \nabla f_1,\nabla f_2\rangle\rangle }\ \mathfrak{m}\text{-a.e.}
	\end{equation}
	Moreover, the following Bochner formula holds for any $f,\varphi\in \mathrm{Test}(X)$ with $\varphi\geqslant 0$.

	\begin{equation}\label{abc2.14}
		\frac{1}{2}\int_{X}  \Delta \varphi \cdot |\nabla f|^2\,\mathrm{d}\mathfrak{m}\geqslant \int_{X}\varphi \left(|\mathop{\mathrm{Hess}}f|^2+ \langle \nabla \Delta f,\nabla f\rangle+K|\nabla f|^2\right) \mathrm{d}\mathfrak{m}.
	\end{equation}

\end{thm}

\begin{remark}\label{rmk2.16}
	Since $\mathrm{Test}(X)$ is dense in $D(\Delta)$ with respect to the norm $\sqrt{\|\cdot\|_{H^{1,2}}^2+\|\Delta(\cdot)\|_{L^2}^2}$, even for $f\in D(\Delta)$, (\ref{abc2.14}) guarantees that $\mathop{\mathrm{Hess}}f$ is well-defined and belongs to $ L^2\left((T^\ast)^{\otimes 2}X\right)$. 
	
	For any $x\in X$, by \cite[Theorem 6.7]{AMS14}, \cite{G18b} and \cite[Lemma 3.1]{MN14}, there exists a cut-off function $\varphi \in \mathrm{Test}(X)$ such that $0\leqslant\varphi\leqslant 1$, $\varphi=1$ on $B_r(x)$, and $\varphi\equiv 0$ outside $ B_{2r}(x)$. Furthermore, it satisfies  $r^2|\Delta \varphi|+r|\nabla \varphi|\leqslant C(K,N)$. Combining (\ref{abc2.14}) and integration by parts, for any $r>0$, the local $L^2$-norm of the Hessian of $f\in D(\Delta)$ can be estimated as follows:
	
	\begin{equation}\label{eeeeeeqn2.6}
		{		\begin{aligned}
				r^2	\int_{B_r(x)}\left|\mathrm{Hess} f\right|^2\,\mathrm{d}\mathfrak{m}\leqslant C(K,N)\left(r\int_{B_{2r}(x)}\left({(\Delta f)}^2+{|\nabla f|}^2\right)\,\mathrm{d}\mathfrak{m}+\inf_{m\in\mathbb{R}}\int_{B_{2r}(x)}\left||\nabla f|^2-m\right|\,
				\mathrm{d}\mathfrak{m}\right).
		\end{aligned}}
	\end{equation}
	%
	%
	%
	
\end{remark}
\begin{remark}\label{rmk2.7}
	By \cite[Proposition 3.3.24]{G18b},  the Hessian also admits the locality property: for any $f_i\in D(\Delta)$ $(i=1,2)$, $|\mathrm{Hess}\, (f_1-f_2)|=0$ $\mathfrak{m}$-a.e. on $\{f_1=f_2\}$. Thus even if a function $f$ belongs to $D(\Delta,B_{2r}(x))$, by using the cut-off function $\varphi$ from Remark \ref{rmk2.16}, the Hessian of $f$ on $B_r(x)$ can be interpreted as $\mathrm{Hess}(\varphi f)$. If moreover $f_i\in \mathrm{Lip}(X,\mathsf{d})$, then by \cite[Proposition 3.3.22]{G18b} we have $\langle\nabla f_1,\nabla f_2\rangle\in H^{1,2}$ and
	\begin{equation}\label{11eqn2.16}
		\nabla \langle \nabla f_1,\nabla f_2 \rangle =\mathop{\mathrm{Hess}}f_1(\nabla f_2,\cdot)+ \mathop{\mathrm{Hess}}f_2(\nabla f_1,\cdot).
	\end{equation}
\end{remark}
\begin{defn}[Covariant derivative]For a vector field $V \in L^2(TX)$, we say $V \in W_{C}^{1,2}(TX)$ if there is tensor $T \in L^2((T^*)^{\otimes 2}X)$ such that for any $f_1, f_2 \in \text{Test}(X)$,
	\[
	T(\nabla f_1, \nabla f_2) = \left\langle \nabla \langle V, \nabla f_1 \rangle, \nabla f_2 \right\rangle - \text{Hess } f_2(V, \nabla f_1).
	\]
	Such $T$ is called the \textit{covariant derivative} of $V$, and is denoted by $\nabla V$. The $W_{C}^{1,2}$-norm of $V$ is defined by
	\[
	\|V\|_{W_{C}^{1,2}(TX)}^2 := \|V\|_{L^2(TX)}^2 + \|\nabla V\|_{L^2((T^*)^{\otimes 2}X)}^2.
	\]
	
\end{defn}
Denote by $H_{C}^{1,2}(TX)$ the $W_{C}^{1,2}$-closure of $\text{TestV}(X)$, where
\[
\mathrm{TestV}(X)=\mathrm{TestV}(X,\mathsf{d},\mathfrak{m}):=\left\{
\sum_{i=1}^{k} f_i \nabla g_i:\ f_i,g_i\in \mathrm{Test}(X)
\right\}.
\]
We remark that in general $W_{C}^{1,2}(TX) \neq H_{C}^{1,2}(TX)$.

\begin{defn}[Divergence]We denote by $D(\operatorname{div})$ the space of all vector fields $V \in L^2(TX)$ for which there exists $h \in L^2$ satisfying
	
	\[
	-\int f \, h \, \mathrm{d}\mathfrak{m} = \int d\,f(V) \, \mathrm{d}\mathfrak{m},\  \forall f \in H^{1,2}(X).
	\]
	This function $h$ is denoted by $\operatorname{div}(V)$. It is unique by the density of $H^{1,2}$ in $L^2$. Moreover, $D(\operatorname{div})$ is a vector subspace of $L^2(TX)$ and $\operatorname{div} : D(\operatorname{div}) \to L^2$ is a linear operator.
\end{defn}
\begin{remark}
	The divergence also enjoys the following locality property: for any $V,W\in D(\mathrm{div})$, $\mathrm{div}(V)=\mathrm{div}(W)$ $\mathfrak{m}\text{-a.e. on } \{V=W\}$. Moreover, for every bounded Lipschitz function $f$, it holds that
	\[
	\mathrm{div}(fV)=d\, f(V)+f\mathrm{div}(V)\ \mathfrak{m}\text{-a.e}.	
	\]
\end{remark}
\subsection{Strong locality of Laplacian and non-degeneration of the heat kernel gradient}
The following proposition is the main tool in this paper. Using it, we show that the gradient of the heat kernel is non-zero almost everywhere in this subsection.
\begin{thm}[Strong locality of Laplacian]\label{str}
	Let $(X,\mathsf{d},\mathscr{H}^n)$ be a non-collapsed $\mathrm{RCD}(K,n)$ space. Then for any $g\in \mathrm{Lip}_{\mathrm{loc}}(X,\mathsf{d})\cap D(\Delta)$, $\mathrm{Hess}\, g=0$  $\mathscr{H}^n$-a.e. on $\{|\nabla g|=0\}$. In particular, $\Delta g=0$ $\mathscr{H}^n$-a.e. on $\{|\nabla g|=0\}$.
\end{thm}
\begin{proof}
	Let $A:=\{|\nabla g|=0\}$. We only consider the case $\mathscr{H}^n(A)>0$. For any $u\in \mathrm{Test}(X)$, since $ |\langle\nabla u,\nabla g\rangle|\leqslant |\nabla u||\nabla g|=0$ $\mathscr{H}^n$-a.e. on $A$, the locality of Sobolev functions implies that $|\nabla \langle\nabla u,\nabla g\rangle|=0$ $\mathscr{H}^n$-a.e. on $A$. Moreover, by (\ref{11eqn2.16}) we have
	\[
	|\mathrm{Hess}\,g(\nabla u,\cdot)|\leqslant |\nabla \langle\nabla u,\nabla g\rangle|+|\mathrm{Hess}\,u| |\nabla g|=0,\ \mathscr{H}^n\text{-a.e. on } A.
	\]
Therefore $\mathrm{Hess}\, g=0$ $\mathscr{H}^n$-a.e. on $A$. The second statement follows directly from Proposition \ref{prop1.1}.
\end{proof}
We now recall some basic knowledge of heat kernel on RCD spaces. For details, see \cite{St95,St96,JLZ16,ZZ19}. 

Every RCD$(K,N)$ space $(X,\mathsf{d},\mathfrak{m})$ possesses a unique locally Lipschitz continuous heat kernel. More precisely, there exists a non-negative function $\rho$ on $X \times X \times (0,\infty)$ such that the unique solution to the heat equation can be expressed as follows.
\[
\mathrm{h}_t f=\int_X \rho(\cdot,y,t)f(y)\,\mathrm{d}\mathfrak{m}(y),\ \forall f\in L^2,\  \forall t>0,
\]
where by solution to the heat equation we mean $\{\mathrm{h}_t f\}_{t>0}\subset H^{1,2}\cap D(\Delta)$ solves
 \begin{equation}\label{2.7}
\frac{d }{d t} \mathrm{h}_t f=\Delta \mathrm{h}_t f\ \mathrm{in }\ L^2;\ \  \lim_{t\downarrow 0}\| \mathrm{h}_t f-f\|_{L^2}=0.
\end{equation}
If moreover $f\in H^{1,2}$, then because Cheeger energy is lower semicontinuous and $t\mapsto \int_{X}|\nabla \mathrm{h}_t f|^2\,\mathrm{d}\mathfrak{m}$ is non-increasing, we have $\lim_{t\downarrow 0}\|\mathrm{h}_tf-f\|_{H^{1,2}}=0$. 

Owing to the stochastic completeness of the space, the heat semigroup $\{\mathrm{h}_t\}_{t>0}$ extends naturally to a contraction semigroup on $L^\infty\cap L^2$ in the sense that
\[
\|\mathrm{h}_t f\|_{L^\infty}\leqslant \|f\|_{L^\infty},\ \forall f\in L^\infty\cap L^2.
\]

\begin{thm}[Gaussian estimates for the heat kernel {\cite[{}{Theorems 1.1 and 1.2}]{JLZ16}}]\label{JLZ}
	Let $\rho$ be the heat kernel of an $\mathrm{RCD}(K,N)$ space $(X,\mathsf{d},\mathfrak{m})$. Then given any $\varepsilon>0$, there exists {}{a constant $C=C(K,N,\varepsilon)$ such that
		\begin{equation}\label{JLZineq}
			\frac{1}{C\,\mathfrak{m}(B_{\sqrt{t}}(x))}\exp\left(-\frac{\mathsf{d}^2\left(x,y\right)}{(4-\varepsilon )t}-C t\right)\leqslant \rho(x,y,t)\leqslant \frac{C}{\mathfrak{m}(B_{\sqrt{t}}(x))}\exp\left(-\frac{\mathsf{d}^2\left(x,y\right)}{(4+\varepsilon )t}+C t\right)
		\end{equation}
		holds for any $x,y\in X$ and
		\begin{equation}\label{JLZineq2}
			|\nabla_x \rho(x,y,t)|\leqslant \frac{C}{\sqrt{t}\mathop{\mathfrak{m}(B_{\sqrt{t}}(x))}}\exp\left(-\frac{\mathsf{d}^2\left(x,y\right)}{(4+\varepsilon)t}+C t\right).
		\end{equation}	}
	\end{thm}
	
		\begin{thm}[{\cite[Theorem 3]{D97}}]\label{Davis}
			The heat kernel $\rho(x,y,t)$ is a real analytic function for $t>0$ and satisfies
			\begin{align}\label{fheuiohafohefoiuahfo}
			\left|\frac{\partial^m}{\partial t^m}\rho(x,y,t)\right|\leqslant \frac{m!}{(t-s)^m}\sqrt{\rho(x,x,s)\rho(y,y,s)}, \ \forall s\in (0,t),\ \forall m\in\mathbb{N},\ \forall x,y\in X.
			\end{align}
		\end{thm}

For a bounded open subset $U\subset X$ with $\mathrm{diam}(U)< \mathrm{diam}(X)$, let $H^{1,2}_0(U)$ be the $H^{1,2}$-closure of $\mathrm{Lip}_{\mathrm{c}}(U,\mathsf{d})$. This space is dense in $L^2(U)$ and is a Hilbert space. Analogous to (\ref{2.7}), following a straightforward modification of \cite[Proposition 5.2.4]{G18a}, one can introduce the Dirichlet heat flow $\{\mathrm{h}^U_t f\}_{t> 0}\subset H^{1,2}_0(U)\cap D(\Delta,U)$ for $f\in L^2(U)$, which is defined as the unique solution to
\[
 \frac{d }{d t} \mathrm{h}_t^U f=\Delta \mathrm{h}_t^U f\ \mathrm{in }\ L^2(U,\mathfrak{m}),\ \forall t>0;\ \  \lim_{t\downarrow 0}\| \mathrm{h}_t^U f-f\|_{L^2(U,\mathfrak{m})}=0.
\]
Moreover, there exists a unique Dirichlet heat kernel $\rho^U(x,y,t):U\times U\times (0,\infty)\to [0,\infty)$ such that the Dirichlet heat flow can be expressed as follows.
\[
\mathrm{h}_t^U f=\int_U \rho^U(\cdot,y,t)f(y)\,\mathrm{d}\mathfrak{m}(y),\ \forall f\in L^2(U,\mathfrak{m}),\  \forall t>0.
\]

Note that the weak maximum principle implies the monotonicity of Dirichlet heat
kernels with respect to domains. Namely, for two bounded domains $U\subset U'\subset X$, we have
\[
\rho^U(x,y,t)\leqslant \rho^{U'}(x,y,t),\ \forall (x,y,t)\in U\times U\times (0,\infty).
\]
Together with (\ref{JLZineq}), this yields 
\[
\rho^U(x,y,t)\leqslant \frac{C}{\mathfrak{m}(B_{\sqrt{t}}(x))}\exp\left(-\frac{\mathsf{d}^2\left(x,y\right)}{(4+\varepsilon )t}+C t\right),\  \forall (x,y,t)\in U\times U\times (0,\infty).
\]
\begin{remark}\label{rmk2.23}
		Since Theorem \ref{Davis} also holds for $\rho^U$, combining it with \cite{J14} gives $\rho(x,\cdot,t)\in \mathrm{Lip}_{\mathrm{loc}}(U,\mathsf{d})$ for any $x\in U$ and $t>0$. By a standard Laplacian comparison theorem \cite{G15} and Li-Yau type Harnack inequality \cite{Ch99,GM142} argument as in \cite{JLZ16}, \cite[Page 158-159]{St92}, $\rho^U$ is positive on $U\times U\times (0,\infty)$.
\end{remark}

We next recall the following result. For the discreteness of eigenvalues, we refer to \cite{GMS15} for a standard approach and to \cite[Section 2.2]{ZZ19} for a survey on Dirichlet eigenvalues. Regarding estimates of eigenvalues and eigenfunctions, see \cite[Section 3.2]{ZZ19} and \cite[Appendix]{AHPT21}.

\begin{prop}\label{prop3.1}Let $(X, \mathsf{d}, \mathfrak{m})$ be an $\mathrm{RCD}(K, N)$ space.
	\begin{enumerate}
		\item[$(1)$] Let $U,V$ be bounded open sets such that $U\Subset V$. The Dirichlet eigenvalues of the Laplacian on $U$ are discrete and are denoted by 
		\[
		0 < \mu_1^U \leqslant \mu_2^U \leqslant \cdots
		\]
		counted with multiplicities. Moreover, if we let $\phi_i^U$ be the corresponding eigenfunction normalized by $\|\phi_i^U\|_{L^2(U)} = 1$, then $\{\phi_i^U\}_{i=1}^\infty$ forms an $L^2(U,\mathfrak{m})$-orthonormal basis.
		\item[$(2)$] If $X$ is compact, the eigenvalues of the Laplacian are discrete and are denoted by 
		\[
		0 = \mu_0 < \mu_1 \leqslant \mu_2 \leqslant \cdots
		\]
		counted with multiplicities. Moreover, if we let $\phi_i$ be the corresponding eigenfunction normalized by $\|\phi_i\|_{L^2} = 1$, then $\{\phi_i\}_{i=0}^\infty$ forms an $L^2(X,\mathfrak{m})$-orthonormal basis.
			\end{enumerate}
	In both cases, there exists constants $C_1=C_1(K,N,\mathsf{d}(U,V^c))$ and $C_2= C_2(K,N,\mathrm{diam}(X),\mathfrak{m}(X))$ such that for all $i \geqslant 1$ we have
		\[
	\|\phi_i^U\|_{L^\infty(U)}\leqslant C_1\,{(\mu_i^U)}^{\frac{N}{4}},\ \| \nabla \phi_i^U \|_{L^\infty(U)}\leqslant C_1\,{(\mu_i^U)}^{\frac{N+2}{4}},\ {C_1}^{-1} \,i^{\frac{2}{N}}\leqslant \mu_i^U\leqslant C_1\,i^2.
	\]
	\[
	\|\phi_i\|_{L^\infty}\leqslant C_2\,\mu_i^{\frac{N}{4}},\ \| \nabla \phi_i \|_{L^\infty}\leqslant C_2\,\mu_i^{\frac{N+2}{4}},\ {C_2}^{-1} \,i^{\frac{2}{N}}\leqslant \mu_i\leqslant C_2\,i^2.
	\]
\end{prop}

\begin{remark}\label{rmk3.2}
Under the assumptions of Proposition \ref{prop3.1}, by \cite{J14, JLZ16} and \cite[Appendix]{AHPT21}, the heat kernels admit the spectral representations
		\[
\rho^U(x, y, t) = \sum_{i=1}^\infty e^{-\mu_i^U t} \phi_i^U(x) \phi_i^U(y),	\qquad	\rho(x, y, t) = \sum_{i=0}^\infty e^{-\mu_i t} \phi_i(x) \phi_i(y).
		\]

\end{remark}

We are now ready to show the following proposition.
\begin{prop}[Non-degeneration of the heat kernel gradient]\label{1prop3.3}
Suppose $(X,\mathsf{d},\mathscr{H}^n)$ is a non-collapsed $\mathrm{RCD}(K,n)$ space and $U,V$ are bounded open subsets of $X$ such that $U\Subset V$. Denote by $\rho$, $\rho^U$ be the heat kernel of $X$ and the Dirichlet heat kernel of $U$ respectively. Then for any fixed $y\in X$, $z\in U$ and $t>0$, we have 
\[
|\nabla_x \rho(x,y,t)|\neq 0\ \mathscr{H}^n\text{-a.e. }x\in X \ \text{and}\  \ |\nabla_x \rho^U(x,z,t)|\neq 0\ \mathscr{H}^n\text{-a.e. }x\in U. 
\]
\end{prop}
\begin{proof}
	Since the proofs are almost the same, we only prove the result for $\rho$ by contradiction. Assume there exists a Borel set $A\subset X$ such that $\mathscr{H}^n(A)>0$ and $|\nabla\rho(\cdot,y,t)|=0$ on $A$. Then by Theorem \ref{stronglocality}, we have
	\[
	\Delta \rho(\cdot,y,t)=\frac{\partial}{\partial t}\rho(\cdot,y,t)=0\ \mathscr{H}^n\text{-a.e. on}\ A.
	\]Similarly, we get
	\[
		\Delta^{(m)} \rho(\cdot,y,t)=\frac{\partial^m}{\partial t^m}\rho(\cdot,y,t)=0\ \mathscr{H}^n\text{-a.e. on}\ A,\ \forall m\in\mathbb{N}.
	\]
	
	According to (\ref{JLZineq})-(\ref{fheuiohafohefoiuahfo}) and \cite[Theorem 1.1]{J14}, $\frac{\partial^m}{\partial t^m}\rho(\cdot,y,t)$ is locally Lipschitz continuous for every $m\in \mathbb{N}$. Thus Theorem \ref{Davis} implies that 	$t\mapsto\rho(x,y,t)$ is a constant function for every $x\in \bar{A}$.
	
	Let us fix a point $x\in \bar{A}\setminus\{y\}$. Owing to Theorem \ref{BGineq}, one has
	\[
	\frac{\mathscr{H}^n(B_{\sqrt{t}}(x))}{V_{K,n}(\sqrt{t})}\geqslant \frac{\mathscr{H}^n(B_{1}(x))}{V_{K,n}(1)},\ \forall t\in (0,1),
	\]
	which yields
	\[
	\rho(x,y,t)\leqslant \lim_{t\to 0} \frac{C}{\mathscr{H}^n(B_{\sqrt{t}}(x))}\exp\left(-\frac{\mathsf{d}\left(x,y\right)}{(4+\varepsilon)t}+C t\right)\leqslant 	\lim_{t\to 0} C\, t^{-\frac{n}{2}}\exp\left(-\frac{\mathsf{d}\left(x,y\right)}{(4+\varepsilon)t}\right)=0.
	\]
This is a contradiction because $\rho(x,y,t)$ is positive.

	\end{proof}

Next, we give some results concerning the asymptotic behavior of the heat kernel and non-degeneration of the gradient of the heat flow on special sets determined by the initial value of the heat equation and the compactness of the space.
\begin{lem}\label{1cor3.7}
	Suppose $(X,\mathsf{d},\mathscr{H}^n)$ is a compact non-collapsed $\mathrm{RCD}(K,n)$ space with $\mathscr{H}^n(X)=1$. Then for every $x,y\in X$ it holds that
	\[
	\lim_{t\to\infty} \rho(x,y,t)=1\ \text{and}\ \lim_{t\to\infty} \|\nabla \rho(x,\cdot,t)\|_{L^\infty}=0.
	\] 
\end{lem} 
\begin{proof}
	By Proposition \ref{prop3.1}, we have
	\begin{equation}\label{feiaohhfoeaihfeia}
		\sum_{i=1}^{\infty}\exp(-\mu_i t)|\varphi_i(x)\varphi_i(y)|\leqslant C_1\sum_{i=1}^{\infty}\exp\left(-{C_1}^{-1} i^{\frac{2}{n}} t\right) i^n\to 0\ \  \text{as}\ t\to \infty.
	\end{equation}
	Regarding the second limit, one may let $r=\mathrm{diam}(X)/100$, and then apply \cite[Theorem 1.1]{J14} and Theorem \ref{BGineq} to obtain
	\small{\[
\|\nabla \rho(x,\cdot,t)\|_{L^\infty(B_r(y))}\leqslant C(K,n)\fint_{B_{2r}(y)}|\rho(x,\cdot,t)-1|\,\integral\leqslant C(K,n,\mathrm{diam}(X))\int_{B_{2r}(y)}|\rho(x,\cdot,t)-1|\,\integral.
	\]}From the dominating convergence theorem we know the rightmost term of the above inequality converges to $0$ as $t\to \infty$. Because $X$ is compact, we complete the proof.
\end{proof}
\begin{prop}\label{prop3.4}
Under Lemma \ref{1cor3.7}, for any $f\in L^2$ and any $t>0$, $|\nabla \mathrm{h}_t f|\neq 0$ $\mathscr{H}^n$-a.e. on$\ \{f\neq \int_X f\,\mathrm{d}\mathscr{H}^n\}$.
\end{prop}

\begin{proof}
	We argue by contradiction. Assume there exist $t_0>0$ and a Borel set $A\subset \{f\neq\int_X f\,\mathrm{d}\mathscr{H}^n\}$ such that $\mathscr{H}^n(A)>0$ and $|\nabla \mathrm{h}_{t_0} f|=0$ on $A$. According to the proof of Proposition \ref{1prop3.3}, we know $t\mapsto \mathrm{h}_t f(x)$ is a constant function for every $x\in A$.
	
	Applying H\"{o}lder's inequality, we see
	\[
	\left(\int_X \rho(x,y,t)|f|(y)\,\mathrm{d}\mathscr{H}^n(y)\right)^2\leqslant \rho(x,x,2t)\|f\|_{L^2}^2,\ \forall x\in X.
	\] 
	This together with Lemma \ref{1cor3.7} and the dominating convergence theorem shows
	\[
	\lim_{t\to \infty} \mathrm{h}_t f(x)=\lim_{t\to \infty}\int_X \rho(x,y,t)f(y)\,\mathrm{d}\mathscr{H}^n(y)=\int_X f\,\mathrm{d}\mathscr{H}^n,\ \forall x\in A,
	\]
	which leads to a contradiction with the definition of $A$. 
\end{proof}
Since a similar estimate to (\ref{feiaohhfoeaihfeia}) also holds for the Dirichlet heat kernel, the following corollary holds.
\begin{cor}
	Assume $(X,\mathsf{d},\mathscr{H}^n)$ is a non-collapsed $\mathrm{RCD}(K,n)$ space and $U,V$ are bounded open subsets of $X$ such that $U\Subset V$. Let $\rho^U$ be the Dirichlet heat kernel of $U$. Then for every $x,y\in U$ and $r\in (0,\mathsf{d}(y,U^c)/10)$ it holds that 
	\[
	\lim_{t \to\infty}\rho^U(x,y,t)=0 \ \text{and}\ 	\lim_{t \to\infty} \|\rho^U(x,\cdot,t)\|_{L^\infty(B_r(y))}=0.
	\] Moreover, for any $f\in L^2(U)$ and any $t>0$, we have $|\nabla \mathrm{h}^U_t f|\neq 0\ \mathscr{H}^n\text{-a.e. on }\{f\neq 0\}$.
\end{cor}
\begin{prop}\label{prop3.9}
	Assume $(X,\mathsf{d},\mathscr{H}^n)$ is a non-compact non-collapsed $\mathrm{RCD}(K,n)$ space. Then for any $f\in L^2$ and any $t>0$, we have
	$|\nabla \mathrm{h}_t f|\neq 0$ $\mathscr{H}^n$-a.e. on$\ \{f\neq 0\}$.
\end{prop}
\begin{proof}
	For any $g\in\mathrm{Test}(X)$, we have
	\[
	\frac{d}{dt}\int_X {(\mathrm{h}_t g)}^2\,\integral=2\int_X \mathrm{h}_t g\frac{d}{dt}\mathrm{h}_t g \,\integral=-2\int_X {|\nabla\mathrm{h}_t g|}^2\,\integral\leqslant 0,
	\]
	so $t\mapsto \|\mathrm{h}_t g\|_{L^2}$ is non-increasing on $(0,\infty)$. Because $t\mapsto \int_X {|\nabla\mathrm{h}_t g|}^2\,\integral$ is also non-increasing on $(0,\infty)$, $\lim_{t\to\infty} \|\mathrm{h}_t g\|_{L^2}=0$. In particular, for every $x\in X$, letting $g=\rho(x,\cdot,1)$ yields
	\begin{align}\label{eqn3.4}
	\lim_{t\to \infty}\rho(x,x,t)=\lim_{t\to \infty} \|\mathrm{h}_t(\rho(x,\cdot,1))\|_{L^2}^2=0.
	\end{align}
According to the proof of Proposition \ref{prop3.4} and (\ref{eqn3.4}), if there exist $t_0>0$ and a Borel set $A\subset \{f\neq 0\}$ with $\mathscr{H}^n(A)>0$ such that $|\nabla \mathrm{h}_{t_0} f|=0$ on $A$, then $\mathrm{h}_t f=0$ on $A$ for any $t\in (0,\infty)$. A contradiction occurs since $\mathrm{h}_t f(x)$ converges to $f(x)$ as $t\to 0$ for $\mathscr{H}^n$-a.e. $x\in A$.
\end{proof}
\begin{cor}
	Under Proposition \ref{prop3.9}, for any $x,y\in X$ and $R>0$ it holds that 
	\[
	\lim_{t\to\infty} \rho(x,y,t)=\lim_{t\to \infty}\|\nabla \rho(x,\cdot,t)\|_{L^\infty(B_R(y))}=0.
	\]
\end{cor}
\begin{proof}
	By \cite[Theorem 1.1]{J14} and the stochastic completeness of heat kernel, for any $z\in X$ we have
	\[
	\|\nabla \rho(x,\cdot,t)\|_{L^\infty(B_1(z))}\leqslant C(K,n)\fint_{B_2(z)}\rho(x,y,t)\,\integral(y)\leqslant \frac{C(K,n)}{\mathscr{H}^n(B_2(x))}.
	\]
By (\ref{eqn3.4}), we know for any $x,z\in X$, $\rho(x,\cdot,t)$ uniformly converges to $0$ on $B_1(z)$ as $t\to \infty$. 
\end{proof}

\section{Sets of finite perimeter on RCD spaces}\label{sec3}
This section is aimed at proving an integral type Gauss–Green formula on non-collapsed RCD spaces. We begin by introducing the notion of sets of finite perimeter on RCD spaces. For further details, see \cite{BPS23b,BPS23}. Let $(X, \mathsf{d}, \mathfrak{m})$ be an $\text{RCD}(K, N)$ space.

\begin{defn}
Let $E \subset X$ be a Borel set and $A \subset X$ be an open set. Define

\[
\operatorname{Per}(E, A) := \inf \left\{ \liminf_{i \to \infty} \int_A |\nabla f_i| \, \mathrm{d}\mathfrak{m} \;\middle|\; f_i \in \operatorname{Lip}(A, \mathsf{d}), \; f_i \to \chi_E \text{ in } L^1(A, \mathfrak{m}) \right\}.
\]

If $\operatorname{Per}(E, X) < \infty$, we say $E$ is a set of \textit{finite perimeter}. If $\operatorname{Per}(E, \Omega) < \infty$ for any open set $\Omega \Subset X$, we say $E$ has \textit{locally finite perimeter}.
\end{defn}

For a set $E \subset X$ of locally finite perimeter, the map $\operatorname{Per}(E, \cdot)$ can be extended to an outer measure, which we will still denote by $\operatorname{Per}(E, \cdot)$. Let us also define

\[
E^{(t)} := \left\{ x \in X \;\middle|\; \lim_{r \downarrow 0} \frac{\mathfrak{m}(E \cap B_r(x))}{\mathfrak{m}(B_r(x))} = t \right\}.
\]
It was proved in \cite{BPS23b} that $\operatorname{Per}(E, \cdot)$ is concentrated on $\mathcal{F}E := E^{(1/2)}$, which is called the \textit{essential boundary} of $E$.

Analogous to the notion of essential dimension for RCD spaces, there is also a counterpart for boundaries of sets of finite perimeter. We introduce the definition of regular boundary points as follows.

\begin{defn}
	Let $E \subset X$ be a set of locally finite perimeter. A quintuple $(Y, \mathsf{d}_Y, \mathfrak{m}_Y, y, F)$ is called a \textit{tangent space} to $E$ at $x \in X$ if there exists $r_i\downarrow 0$ and a Borel set $F\subset Y$ such that 
	\[
	\left(X, {r_i}^{-1}\mathsf{d}, {(\mathfrak{m}(B_{r_i}(z)))}^{-1}\mathfrak{m}, x\right) \xrightarrow{\text{pmGH}} (Y, \mathsf{d}_Y, \mathfrak{m}_Y, y),
	\]
	and under this convergence, $f_i = \chi_E$ (viewed as functions on the rescaled spaces) converges to $\chi_F$ in $L^1_{\text{loc}}$-strong sense (see \cite[Definition 3.1]{ABS19} for the precise definition).

The set of $k$-dimensional \textit{regular boundary points} of $E$ is then defined as
	\[
	\mathcal{F}_k E := \{z \in X \mid \text{The tangent space to } E \text{ at } z \text{ is } (\mathbb{R}^k, \mathsf{d}_{\mathbb{R}^k}, {\omega_n}^{-1}\mathscr{L}^k, 0_k, \{x_k > 0\})\}.
	\]
\end{defn} By \cite{BPS23}, for $n=\mathrm{dim}_{\mathsf{d},\mathfrak{m}}(X)$, one has $\mathrm{Per}(E,\partial E\setminus \mathcal{F}_n(E))=\mathrm{Per}(E,\partial E\setminus \mathcal{F}E)=0$.

We now recall the co-area formula in the non-smooth setting. Details can be found in \cite[Proposition 4.2]{M03}. See also \cite[Remark 3.5]{GH16} for the coincidence of the total variation measure of a Lipschitz function with its weak upper gradient multiplying the reference measure.

\begin{thm}\label{coareafor}For every $v\in \mathrm{Lip}(X,\mathsf{d})$, $\{v > t\}$ has locally finite perimeter for $\mathscr{L}^1$-a.e. $t\in \mathbb{R}$. Moreover, for every Borel function $f: X \rightarrow [0, \infty]$ it holds:
	\[
	\int_{\{s\leqslant v<t \}} f \,|\nabla v| \, \mathrm{d}\mathfrak{m} = \int_{s}^t \left( \int_X f \, \mathrm{d}\mathrm{Per}(\{v>r\},\cdot) \right) \mathrm{d}r,\ \forall -\infty\leqslant s<t\leqslant \infty.
	\]
\end{thm}
Let us recall the notion of Sobolev capacity and quasi-continuity. For any Borel set $E \subset X$, the \textit{Sobolev capacity} $\operatorname{Cap}(E)$ is defined by

\[
\operatorname{Cap}(E) := \inf \left\{ \|u\|_{H^{1,2}(X,\mathsf{d},\mathfrak{m})} \mid u \equiv 1 \ \mathrm{on}\  E\right\}.
\]
We say a function $f$ is \textit{quasi-continuous} if for any $\varepsilon > 0$, there exists $E_\varepsilon \subset X$ with $\operatorname{Cap}(E_\varepsilon) < \varepsilon$ such that $f|_{X \setminus E_\varepsilon}$ is continuous. It is known that every $H^{1,2}$ function is quasi-continuous; see for instance \cite[Section 5]{BB11}.

The notion of "restriction to the boundary" of quasi-continuous functions and vector fields was introduced in \cite{BPS23}. Let us fix a set $E \subset X$ of locally finite perimeter. Since $\operatorname{Per}(E, \cdot) \ll \operatorname{Cap}$, the restriction of a quasi-continuous function to $\partial E$ is well-defined in the $\operatorname{Per}(E, \cdot)$-a.e. sense. We denote by $\operatorname{tr}_E f$ this restriction of $f$ to $\partial E$. With the help of $\operatorname{tr}_E$, we can define $L^2$-vector fields on the boundary of sets of finite perimeter:
\begin{thm}[{\cite[Theorem 2.2]{BPS23}}] There exists a unique couple $(L^2_E(TX), \bar{\nabla})$ where $L^2_E(TX)$ is an $L^2(\operatorname{Per}(E, \cdot))$-normed $L^\infty(\operatorname{Per}(E, \cdot))$-module and $\bar{\nabla} : \operatorname{Test}(X) \to L^2_E(TX)$ is linear satisfying
	\begin{enumerate}
		\item[$(1)$]  For any $f \in \operatorname{Test}(X)$, $|\bar{\nabla}f| = \operatorname{tr}_E |\nabla f|$ $\operatorname{Per}(E, \cdot)$-a.e.
		\item[$(2)$]The set of test vectors
		\[
		\operatorname{TestV}_E(X) := \left\{ \sum_{i=1}^k \chi_{E_i} \bar{\nabla} f_i \;\middle|\; \{E_i\}_{i=1}^k \text{ is a Borel partition of } X,\; f_i \in \operatorname{Test}(X) \right\} 
		\]
		is dense in $L^2_E(TX)$.
	\end{enumerate}
\end{thm}

Thanks to the counterpart of quasi-continuity for $L^2(TX)$-vector fields introduced in \cite{DGP21}, a similar boundary restriction (or trace) operation can be obtained for $W^{1,2}_C(X)$-vector fields. With a slight abuse of notation, for $V \in W^{1,2}_C(X)$, we also denote its restriction to $\partial E$ by $\operatorname{tr}_E V$. Using the language of the trace introduced above, \cite{BPS23} proved the following Gauss-Green formula (also referred to as the divergence theorem):
\begin{thm}[Gauss-Green formula]\label{thm3.5}There exists a unique $\nu_E \in L^2_E(TX)$ such that $|\nu_E| = 1$ holds $\operatorname{Per}(E, \cdot)$-a.e. and

\[
\int_E \mathrm{div}(V)\,\mathrm{d}\mathfrak{m} = - \int_{\mathcal{F}E} \langle \operatorname{tr}_E V, \nu_E \rangle \, \mathrm{dPer}(E, \cdot),\ \forall V \in H_C^{1,2}(TX) \cap D(\operatorname{div}) \cap L^\infty(TX).
\]
\end{thm}
\subsection{Gauss-Green formula for level sets of gradient non-vanishing functions}
In this subsection, we give a generalization of \cite[Theorem 6.1]{BPS23b}. Let us consider
\begin{itemize}
\item $(X,\mathsf{d},\mathfrak{m})$ is an $\mathrm{RCD}(K,N)$ space, with essential dimension $n\leqslant N$.
\item $f\in H^{1,2}_{\mathrm{loc}}(X,\mathsf{d},\mathfrak{m})\cap\mathrm{Lip}_{\mathrm{loc}}(X,\mathsf{d})\cap D(\Delta)$ with $|\nabla f|>0$ $\mathfrak{m}$-a.e. 
\item For any $t$ in the range of $f$ (except for the maximum), define $\Omega_t:=\{x\in X\mathop |f(x)<t\}$.
\end{itemize}
We will show that, in this case, the outward normal vector field of $\Omega_t$ can be represented explicitly by $-\nabla f/|\nabla f|$ (after taking the trace). Firstly let us consider the behavior of $f$ under blowing up.
%
\begin{prop}\label{prop:blow.up.bx}
Assume $x$ is both a Lebesgue point of ${|\nabla f|}^2$ and ${(\Delta f)}^2$ such that $|\nabla  f|(x)\neq 0$, and for some $r_i\downarrow 0$ we have 
\[
	(X_i,\mathsf{d}_i,\mathfrak{m}_i,x_i):=\left(X, {r_i}^{-1}\mathsf{d},{(\mathfrak{m}(B_{r_i}(x))}^{-1}\mathfrak{m},x\right)\xrightarrow{\mathrm{pmGH}}(\mathbb{R}^n,\mathsf{d}_{\mathbb{R}^n},{\omega_n}^{-1}\mathscr{L}^n,0_n).
\]
Let $f_i:={r_i}^{-1}( f- f(x))$, then $f_i$ converges to a linear function $f_\infty:\mathbb{R}^n\to \mathbb{R}$ in $H^{1,2}_{\mathrm{loc}}$-strong sense after passing to a subsequence. In particular, $|\nabla_{\mathbb{R}^n} f_\infty|\equiv|\nabla  f|(x)$. 
\end{prop}
\begin{proof}
We use $\nabla_i$, $\Delta_i$ to denote the gradient and Laplacian on $(X_i,\mathsf{d}_i,\mathfrak{m}_i,x_i)$ respectively. It is trivially checked that
\[
	|\nabla_i f_i|=|\nabla_X  f|,\ \Delta_if_i=r_i\Delta_X f.
\]
By Theorem \ref{AH18}, passing to a subsequence, we may assume $f_i$ converges to some $f_\infty:\mathbb{R}^n\to \mathbb{R}$ in $H^{1,2}_{\mathrm{loc}}$-strong sense. In particular, for any $R>0$, we have
\[
	\begin{aligned}
		{\omega_n}^{-1}\int_{B_R(0_n)}{(\Delta f)}^2\,\mathrm{d}\mathscr{L}^n
		\leqslant\lim_{i\to\infty}\int_{B_{R}^{X_i}(x_i)}{(\Delta_i f_i)}^2\ \,\mathrm{d}\mathfrak{m}_i=\lim_{i\to\infty}{r_i}^2\,\frac{\mathfrak{m}(B_{r_iR}(x))}{\mathfrak{m}(B_{r_i}(x))}\fint_{B_{r_i R}^X(x)}{(\Delta f)}^2\,\mathrm{d}\mathfrak{m}=0.
	\end{aligned}
\]
Since $|\nabla f|$ is bounded, $|\nabla_{\mathbb{R}^n}f_\infty|$ is also bounded in $\mathbb{R}^n$, i.e., $f_\infty$ has sublinear growth, and thus, combining the harmonicity, is linear. Moreover by the $H^{1,2}_{\mathrm{loc}}$-strong convergence,
\[
	|\nabla_{\mathbb{R}^n} f_\infty|\equiv\fint_{B_1(0_n)}|\nabla_{\mathbb{R}^n} f_\infty|\,\mathrm{d}\mathscr{L}^n=\lim_{i\to\infty}\fint_{B_{1}^{X_i}(x)}{|\nabla _i f_i|}^2\,\mathrm{d}\mathfrak{m}_i=\lim_{i\to\infty }\fint_{B_{r_i}^X(x)}{|\nabla f|}^2\,\mathrm{d}\mathfrak{m}=|\nabla f|(x).
\]
\end{proof}
Next we consider the blow-up along the boundary of $\Omega_t$. By co-area formula and Theorem \ref{BS}, up to an $\mathscr{L}^1$-null set, we may assume that $\Omega_t$ has locally finite perimeter with $\mathrm{Per}(\Omega_t,\partial \Omega_t \setminus A)=0$, where $A$ is the union of $\mathcal{R}_n$, the set of Lebesgue points of $\{|\nabla f|\neq 0\}$ and of ${|\nabla f|}^2$, ${(\Delta f)}^2$. 
\begin{prop}\label{prop:normal.vector.Er}
For every such $\Omega_t$, let $\nu_{\Omega_t}\in L^2_{\Omega_t}(TX)$ be the outer normal vector field of $\Omega_t$ defined as in Theorem \ref{thm3.5}. For any $x\in \mathcal F_n\Omega_t$, let $f_\infty$ be the limit of $ f$ in the tangent space as in Proposition \ref{prop:blow.up.bx}, then the blow-up of $(X,\mathsf{d},\mathfrak{m},x,\Omega_t)$ is $(\mathbb{R}^n,\mathsf{d}_{\mathbb{R}^n},{\omega_n}^{-1}\mathscr{L}^n,0_n,H_{ f})$, where $H_{ f}:=\{v\in\mathbb{R}^n\mathop|\langle v,\nabla_{\mathbb{R}^n} f_\infty\rangle\geqslant 0\}$.
In particular, we have
\[\langle\nu_{\Omega_t},\mathrm{tr}_{\Omega_t}\nabla f\rangle=-|\nabla  f|\ \ \mathrm{Per}(\Omega_t,\cdot)\text{-a.e.}\]
\end{prop}
\begin{proof}
From the choice of $t$, we can apply an almost identical blow-up procedure as in the proof of \cite[Theorem 6.1]{BPS23b} to prove the desired results. We here omit the details. 
\end{proof}
And now we are in the position to determine the normal vector field.
\begin{prop}\label{prop:representation.normal.vec}
For any $u\in \mathrm{Test}(X)$, the following holds for $\mathscr{L}^1$-a.e. $t\in f(X)$:
\[	\langle\nabla u,\nu_{\Omega_t}\rangle(x)=-\left\langle\nabla u,\frac{\nabla  f}{|\nabla f|}\right\rangle(x)\ \mathrm{Per}(\Omega_t,\cdot)\text{-a.e.\ }\]
\end{prop}
\begin{proof}
Recall that for any $u,v\in\mathrm{Test}(X)$, we have $\langle\nabla u,\nabla v\rangle\in H^{1,2}(X)$, and thus is quasi-continuous. We then omit the trace operation below when discussing with respect to $\mathrm{Per}(\Omega_t,\cdot)$ measure in order for brevity. Let us take $x\in \mathcal F_n\Omega_t\cap \mathcal{R}_n$, and assume $x$ is a Lebesgue point (with respect to both $\mathfrak{m}$ and $\mathrm{Per}(\Omega_t,\cdot)$) for any concerned function below. 

By \cite[Proposition 3.6]{BPS23b}, there exist $r_0>0$ and harmonic functions $\{v_i\}_{i=1}^n$ on $B_{r_0}(x_0)$ such that the following holds.
\begin{enumerate}
	\item[(1)] For any $1\leqslant i,j\leqslant n$,
	\[\lim_{r\downarrow 0}\fint_{B_r(x)} |\langle\nabla v_i,\nabla v_j\rangle-\delta_{ij}|\,\mathrm{d}\mathfrak{m}=\lim_{r\downarrow 0}\fint_{B_r(x)} |\langle\nabla v_i,\nabla v_j\rangle-\delta_{ij}|\,\mathrm{d}\mathrm{Per}(\Omega_t,\cdot)=0.\]
	\item[(2)] For any $1\leqslant i\leqslant n$,
	\[
		\nu_i:=\lim_{r\downarrow 0}\fint_{B_r(x)}\langle\nabla v_i,\nu_{\Omega_t}\rangle \,\mathrm{d}\mathrm{Per}(\Omega_t,\cdot)=\delta_{in},\ \lim_{r\downarrow 0}\fint_{B_r(x)}|\langle\nabla v_i,\nu_{\Omega_t}\rangle-\nu_i|\,\mathrm{d}\mathrm{Per}(\Omega_t,\cdot)=0.
	\]
\end{enumerate}
Let us define 
\[
\mu_i:=\lim_{r\downarrow 0}\fint_{B_r(x)} \langle\nabla u,\nabla v_i\rangle\,\mathrm{d}\mathfrak{m}=\lim_{r\downarrow 0}\fint_{B_r(x)} \langle\nabla u,\nabla v_i\rangle\,\mathrm{d}\mathrm{Per}(\Omega_t,\cdot)=\langle\nabla u,\nabla v_i\rangle(x),\ i=1,2,\ldots,n.
\]
Via an argument analog to Proposition \ref{prop:blow.up.bx}, for any $r_j\downarrow 0$ such that
\[
	(X_j,\mathsf{d}_j,\mathfrak{m}_j,x_j):=\left(X,{r_j}^{-1}\mathsf{d},{\mathfrak{m}(B_{r_j}(x))}^{-1}\mathfrak{m},x\right)\xrightarrow{\mathrm{pmGH}}(\mathbb{R}^n,\mathsf{d}_{\mathbb{R}^n},{\omega_n}^{-1}\mathscr{L}^n,0_n),
\]
and that $u_j:={r_j}^{-1}(u-u(x))$, $v_{j}^i:={r_j}^{-1}(v_i-v_i(x))$ $H_{\mathrm{loc}}^{1,2}$-strongly converge to linear functions $u_\infty$ and $v^i_\infty$ respectively in $\mathbb{R}^n$ with $|\nabla_{\mathbb{R}^n} u_\infty|\equiv|\nabla u|(x)$ and $|\nabla_{\mathbb{R}^n}v^i_{\infty}|=1$, we assert that $\{\nabla v_{\infty}^i\}$ are pairwise orthogonal and $\nabla_{\mathbb{R}^n} u_\infty=\sum_i \mu_i\nabla_{\mathbb{R}^n} v^i_\infty$ by the $H^{1,2}_{\mathrm{loc}}$-convergence. Thus $|\nabla u|^2(x)=\sum_i{\mu_i}^2$. Let us consider $\bar u:=\sum_i \mu_i v_i$,
then by the pmGH-convergence, noting that $|\nabla_j u_j|=|\nabla u|$ and $|\nabla_j v_{i,j}|=|\nabla v_i|$, we have
\begin{equation}\label{eq:a1001}
	\begin{aligned}
		&{\langle\nabla (u-\bar u),\nu_{\Omega_t}\rangle}^2(x)=\lim_{r\downarrow 0}\fint_{B_r(x)}{\langle\nabla (u-\bar u),\nu_{\Omega_t}\rangle}^2\,\mathrm{d}\mathrm{Per}(\Omega_t,\cdot)\\
		\leqslant\  &\lim_{r\downarrow 0}\fint_{B_r(x)}{|\nabla (u-\bar u)|}^2\,\mathrm{d}\mathfrak{m}=\lim_{j\to\infty}\fint_{B_{1}^{X_j}(x)}{|\nabla_j(u_j-\bar u_j)|}^2\,\mathrm{d}\mathfrak{m}_j
		={|\nabla_{\mathbb{R}^n}u_\infty|}^2-\sum_i {\mu_i}^2=0,
	\end{aligned}
\end{equation}
where $\bar u_j:=\sum_i\mu_i v^i_{j}$.
Thus $\langle\nabla u,\nu_{\Omega_t}\rangle(x)=\mu_n\langle\nabla v_n,\nu_{\Omega_t}\rangle(x)=\langle\nabla u,\nabla v_n\rangle(x)$.

On the other hand, by Proposition \ref{prop:normal.vector.Er} and \cite[Theorem 6.1]{BPS23b}, letting $f_\infty$ be the limit function of $ f$ as in Proposition \ref{prop:blow.up.bx}, we see $f_\infty$ coincides with $-|\nabla f|(x)v^n_\infty$. In particular, by the $H^{1,2}_{\mathrm{loc}}$-strong convergence,
\[
\langle\nabla  f,\nabla v_i\rangle(x)=\lim_{r\downarrow 0}\fint_{B_r(x)}\langle\nabla  f,\nabla v_i\rangle\,\mathrm{d}\mathfrak{m}=-|\nabla  f|(x)\delta_{in},
\]
\[	
\langle\nabla  f,\nabla \bar u\rangle(x)=-\mu_n|\nabla  f|(x)=-\langle\nabla u,\nu_{\Omega_t}\rangle(x)|\nabla  f|(x).
\]
Similarly as \eqref{eq:a1001}, $\langle\nabla (u-\bar u),\nabla  f\rangle(x)=0$. Therefore,
\[
\langle\nabla u,\nu_{\Omega_t}\rangle(x)=-\left\langle\nabla u,\frac{\nabla  f}{|\nabla f|}\right\rangle(x).
\]
By the fact that $\mathcal F_n\Omega_t$ is $\mathrm{Per}(\Omega_t,\cdot)$-full measured and by the Lebesgue differentiation theorem, the point $x$ chosen at the beginning belongs to a set of full measure with respect to $\mathrm{Per}(\Omega_t,\cdot)$. This completes the proof (For the Lebesgue differentiation theorem with respect to the perimeter, see \cite[Proposition 2.18]{BPS23b}.)
\end{proof}

In sight of Proposition \ref{prop:representation.normal.vec}, by the density of $\mathrm{tr}_{\Omega_t}\mathrm{Test}V(X,\mathsf{d},\mathfrak{m})$ in $L^2_{\Omega_t}(TX)$, we have
\[
\int_{\mathcal F\Omega_t} \langle V,\nu_{\Omega_t}\rangle\,\mathrm{d}\mathrm{Per}(\Omega_t,\cdot)=-\int_{\mathcal F\Omega_t}\left\langle V,\mathrm{tr}_{\Omega_t}\frac{\nabla f}{|\nabla f|}\right\rangle\,\mathrm{d}\mathrm{Per}(\Omega_t,\cdot),\ \forall V\in L^2_{\Omega_t}(TX).\]
That is, $-\nabla f/|\nabla f|$ can be treated as a representative of $\nu_{\Omega_t}$. Therefore the Gauss-Green formula of $\Omega_t$ can be written as follows:
\begin{thm}[Gauss-Green formula for $\Omega_t$]\label{thm:Gauss-Green.bx}
For any compactly supported $V\in H^{1,2}_C(X)\cap D(\mathrm{div})\cap L^\infty(TX)$, the following holds for $\mathscr{L}^1$-a.e. $t\in f(X)$.
\[
	\int_{\Omega_t}\mathrm{div} V\,\mathrm{d}\mathfrak{m}=\int_{\{ f=r\}}\mathrm{tr}_{\Omega_t}\left\langle V,\frac{\nabla  f}{|\nabla  f|}\right\rangle\,\mathrm{d}\mathrm{Per}(\Omega_t,\cdot).
\]
\end{thm}
\subsection{Integral type Gauss-Green formula}
We prove Theorem \ref{GGthm} in this subsection, which is recalled as follows.
\begin{thm}[Integral type Gauss-Green formula]\label{GGthm1}
	Let $(X,\mathsf{d},\mathscr{H}^n)$ be a non-collapsed $\mathrm{RCD}(K,n)$ space. Then for any $f\in \mathrm{Lip}(X,\mathsf{d})$, any compactly supported $g\in \mathrm{Test}(X)$, any $h\in H^{1,2}(X,\mathsf{d},\mathscr{H}^n)$ and any $\tau\in C(\mathbb{R})$, the following integral type Gauss-Green formula holds for $(f,g,h,\tau)$, where $\Omega_t:=\{x\in X:f(x)\leqslant t\}$ is the corresponding $t$-sublevel set of $f$. 
	\begin{equation}\label{a1.1}
		\int_{\Omega_t} \tau\circ f\,\langle\nabla f,h\nabla g\rangle\,\mathrm{d}\mathscr{H}^n=\int_{-\infty}^t \tau(s)\int_{ \Omega_s } \mathrm{div}(h\nabla g)\,\mathrm{d}\mathscr{H}^n\mathrm{d}s,\ \forall t\in\mathbb{R}.
	\end{equation}
\end{thm}

To this end, we fix $h\in H^{1,2}(X,\mathsf{d},\mathscr{H}^n)$ and $\tau\in C(\mathbb{R})$. With $h$ and $\tau$ fixed, we will abbreviate (\ref{a1.1}) as the formula for $(f,g)$. Since the proof for the non-compact case (which uses Dirichlet heat kernel and Dirichlet eigenfunctions) is almost parallel, we restrict ourselves to the case where $X$ is compact. We continue using the notations in Proposition \ref{prop3.1} and Remark \ref{rmk3.2}. We begin with the following lemmas.

\begin{lem}\label{alem3.4}
Let $g\in \mathrm{Test}(X)$. Assume $\{f_i\}\subset \mathrm{Lip}(X,\mathsf{d})$ converges to some $f \in \mathrm{Lip}(X,\mathsf{d})$ in $H^{1,2}$ and uniformly as $i\to \infty$. If (\ref{a1.1}) holds for each $(f_i,g)$, then (\ref{a1.1}) also holds for $(f,g)$.
\end{lem}
\begin{proof}
	For every $t\in\mathbb{R}$, let $\Omega_{t}^i:=\{f_i\leqslant t\}$. We  claim that
	\begin{align}\label{3.1}
	\lim_{i\to \infty}\mathscr{H}^n(\Omega^i_t\,\triangle\, \Omega_t)=0,\ \mathscr{L}^1\text{-a.e.}\ t\in \mathbb{R}.
	\end{align}
	

	Since $\{f_i\}$ uniformly converges to $f$ as $i\to\infty$, we know
	\[
	\lim_{i\to\infty}\mathscr{H}^n(\Omega^i_t\,\triangle\, \Omega_t)\leqslant \mathscr{H}^n(\{f=t\}),\ \forall t\in \mathbb{R}.
	\]
	Because $\{f=t\}$ and $\{f=s\}$ are disjoint whenever $s\neq t$, the set $\{s\in\mathbb{R}:\mathscr{H}^n(\{f=s\})>0\}$ contains at most countable elements. This proves (\ref{3.1}). Moreover, for every element $s$ of the above set, the locality of Sobolev functions yields that $|\nabla f|=0$ a.e. on $\{f=s\}$.

	Therefore using the $H^{1,2}$-convergence of $\{f_i\}$ and (\ref{3.1}) gives
	\[
	\begin{aligned}
		\ &\lim_{i\to \infty}\left|\int_{\Omega_{t}}\langle\nabla f,
		h\nabla g\rangle\,\mathrm{d}\mathscr{H}^n-	\int_{\Omega_{t}^i}\langle\nabla f_i,
		h\nabla g\rangle\,\mathrm{d}\mathscr{H}^n\right|\\
		\leqslant\ &\lim_{i\to \infty}\left|	\int_{\Omega_{t}}h\,(\langle\nabla f,
		\nabla g\rangle-\langle\nabla f_i,\nabla g\rangle)\,\mathrm{d}\mathscr{H}^n\right|+\lim_{i\to \infty} \int_{\Omega^i_t\,\triangle\, \Omega_t}|h\,\langle\nabla f_i,\nabla g\rangle|\,\mathrm{d}\mathscr{H}^n\\
		\leqslant \ &\|\nabla g\|_{L^\infty}\lim_{i\to \infty}\int_{\Omega_t} |h||\nabla (f_i-f)|\,\integral +\|\nabla g\|_{L^\infty}\int_{\{f=t\}}|h||\nabla f|	 \,\mathrm{d}\mathscr{H}^n\, \equiv 0,\ \forall t\in \mathbb{R}.
	\end{aligned}
	\]
Moreover, by (\ref{3.1}) we have
	 \[
	 \lim_{i\to \infty}\int_{\Omega^i_t} (\Delta g+\langle\nabla g,\nabla h\rangle)\,\mathrm{d}\mathscr{H}^n=\int_{\Omega_t} (\Delta g+\langle\nabla g,\nabla h\rangle)\,\mathrm{d}\mathscr{H}^n,\ \mathscr{L}^1\text{-a.e.}\ t\in\mathbb{R}.
	 \]
	 This completes the proof.
\end{proof}
\begin{lem}\label{alem3.5}
For any $x\in X$ and any $i\in \mathbb{N}$, there exists $j\geqslant i$ such that $\phi_j(x)\neq 0$.
\end{lem}
\begin{proof}
	Suppose there exists $m\in\mathbb{N}$ such that $\phi_j(x)=0$ for all $j\geqslant m$. Then Remark \ref{rmk3.2} implies
	\[
	\rho(x,x,t)=\sum_{i=0}^m \exp(-\mu_i t){\phi_i}^2(x),\ \forall t>0.
	\]
	As a result, we have
\[
	\lim_{t\to 0} \mathscr{H}^n(B_{\sqrt{t}}(x))\rho(x,x,t)=0.
\]
	This contradicts (\ref{JLZineq}). Regarding the non-compact case, we remark that the same argument as in Remark \ref{rmk2.23} for Dirichlet heat kernel gives
	\[
	\lim_{t\to 0} \mathscr{H}^n(B_{\sqrt{t}}(x))\rho^U(x,x,t)>0,
	\]	
which also provides the desired contradiction. 
\end{proof}
\begin{lem}\label{lem3.6}
	Let $f=\sum_{j=0}^m a_j\phi_j$ for some $\{a_j\}_{j=0}^m\subset \mathbb{R}$. Then for every $x\in X$, there exists $r>0$ such that given any $g\in \mathrm{Lip}_{\mathrm{c}}(B_{2r}(x),\mathsf{d})\cap D(\Delta,B_{2r}(x))$, (\ref{a1.1}) holds for $(f,g)$.
\end{lem}
\begin{proof}
By Lemma \ref{alem3.5}, there exists $i\gg m$ such that $\phi_i(x)\neq 0$ and $\mu_i\gg \mu_m$. Without loss of generality we may assume $\phi_i(x)>0$. Let 
\[
\xi:=\frac{\phi_i}{\|\phi_i\|_{L^\infty}+\|\nabla\phi_i\|_{L^\infty}},\, f_\varepsilon:=f+\varepsilon\,\xi,\ \forall \varepsilon>0.
\]
It is obvious that $f_\varepsilon\to f$ in $H^{1,2}$ and Lipschitz as $\varepsilon\to 0$. 

 Let $r=\xi(x)/4$. Because $\xi$ is 1-Lipschitz, $\xi\geqslant \xi(x)/2>0$ on $B_{2r}(x)$. Owing to Lemma \ref{alem3.4}, it suffices to show (\ref{a1.1}) holds for every $(f_\varepsilon,g)$ with $g\in \mathrm{Lip}_{\mathrm{c}}(B_{2r}(x),\mathsf{d})\cap D(\Delta,B_{2r}(x))$.

We claim 
\begin{align}\label{3.3}
	|\nabla f_\varepsilon|\neq 0 \ \mathscr{H}^n\text{-a.e. on}\  B_{2r}(x),\ \forall \varepsilon>0. 
\end{align}

Suppose for some $\varepsilon>0$ there exists $A\subset B_{2r}(x)$ with $\mathscr{H}^n(A)>0$ such that $|\nabla f_\varepsilon|=0$ $\mathscr{H}^n$-a.e. on $A$. Theorem \ref{stronglocality} implies $\Delta f_\varepsilon=0$ $\mathscr{H}^n$-a.e. on $A$. Applying Proposition \ref{prop1.1} and Theorem \ref{stronglocality} then shows
\[
\Delta^{(2l)} f_\varepsilon=\sum_{j=0}^m a_j {\mu_j}^{2l} \phi_j+ {\mu_i}^{2l}\xi=0\ \mathscr{H}^n\text{-a.e. on}\ A,\ \forall l\in\mathbb{N}.
\]
This leads to a contradiction since $ \xi\geqslant \xi(x)/2>0$ on $A$ and $\lim_{l\to\infty}\max_{j}(\mu_j/\mu_i)^{2l}=0$.

Finally, letting $\Omega_t^\varepsilon:=\{f_\varepsilon\leqslant t\}$, and combining (\ref{3.3}) with Theorem \ref{thm3.5} and Proposition \ref{prop:representation.normal.vec}, we obtain
\[
\tau(t)\int_{\partial \Omega_t^\varepsilon}\left\langle\frac{\nabla f_\varepsilon}{|\nabla f_\varepsilon|},h\nabla g\right\rangle\,\mathrm{d}\mathrm{Per}(\{f_\varepsilon\geqslant t\},\cdot)=\tau(t)\int_{\Omega_t^\varepsilon} (\Delta g+\langle\nabla g,\nabla h\rangle)\,\mathrm{d}\mathscr{H}^n,\, \mathscr{L}^1\text{-a.e. } t\in \mathbb{R}.
\]
This together with co-area formula then implies
\[
\int_{\Omega_t^\varepsilon}\tau\,\circ f_\varepsilon\,\langle\nabla f_\varepsilon,h\nabla g\rangle\,\mathrm{d}\mathscr{H}^n=\int_{-\infty}^t\tau(s)\int_{ \Omega_s^\varepsilon}(\Delta g+\langle\nabla g,\nabla h\rangle)\,\mathrm{d}\mathscr{H}^n\mathrm{d}s.
\]
\end{proof}
\begin{lem}\label{lem3.7}
	Under Lemma \ref{lem3.6}, given any $g\in \mathrm{Test}(X)$, (\ref{a1.1}) holds for $(f,g)$.
\end{lem}
\begin{proof}
	For any $x\in X$, by Lemma \ref{lem3.6}, there exists $r=r(x)$ such that for any $g\in \mathrm{Lip}_{\mathrm{c}}(B_{2r}(x),\mathsf{d})\cap D(\Delta,B_{2r}(x))$, (\ref{a1.1}) holds for $(f,g)$. The compactness of $X$ then allows us to extract finite such balls 
	$\{B_{2r_j}(x_j)\}_{j=1}^k$ such that $X\subset \cup_{j=1}^k B_{r_j}(x_j)$. 
	
	For each $j$, take a cut-off function $\varphi_j\in \mathrm{Test}(X)$ such that $\varphi_j=1$ on $B_{r_j}(x_j)$ and $\varphi_j=0$ outside $B_{2r_j}(x_j)$. Given any test function $g$, we have the decomposition $g=\sum_{j=1}^k g_j$, where
	\[
	g_j=\frac{\varphi_j}{\sum_{l=1}^k \varphi_l}\,g,\ j=1,\ldots,k.
	\]
	Since $g_j\in \mathrm{Lip}_{\mathrm{c}}(B_{2r_j}(x_j),\mathsf{d})\cap D(\Delta,B_{2r_j}(x_j))$, the conclusion follows immediately from Lemma \ref{lem3.6}.
\end{proof}
We are now in a position to prove Theorem \ref{GGthm1}.
\begin{proof}[Proof of Theorem \ref{GGthm1}]
	Recall that $\mathrm{h}_t f$ converges to $f$ in $H^{1,2}$ as $t\to 0$. Moreover, by \cite[Theorem 6.2]{AGS14a}, we have
	\[
	\|\nabla \mathrm{h}_t f\|_{L^\infty}\leqslant \exp(-Kt)\,\|\nabla f\|_{L^\infty},\ \forall t>0.
	\]
	In view of Lemma \ref{alem3.4}, it suffices to prove that $(\ref{a1.1})$ holds for each $(\mathrm{h}_t f,g)$ with $g\in \mathrm{Test}(X)$. 
	
	Assume $f=\sum_{i=0}^\infty a_i \phi_i$. Then $\mathrm{h}_tf=\sum_{i=0}^{\infty}a_i \exp(-\mu_i t)\phi_i$. Define $f_i:=\sum_{j=0}^i a_j \exp(-\mu_j t)\phi_j$. Since $\sum_{j=0}^\infty {a_j}^2=\|f\|_{L^2}^2$, we may assume $|a_j|\leqslant 1$ for sufficiently large $j$. Proposition \ref{prop3.1} implies, 
	\[
	\|\nabla (f_i-\mathrm{h}_t f)\|_{L^\infty}\leqslant C\sum_{j=i+1}^\infty |a_j|\exp\left(-C^{-1}j^{\frac{2}{n}}t\right)j^{\frac{n+2}{2}}\to 0,\ \mathrm{as}\ i\to\infty.
	\]
	In particular, $\sup_i \|\nabla f_i\|_{L^\infty}<\infty$.
	Thus $\{f_i\}\to \mathrm{h}_t f$ as $i\to \infty$ in both $H^{1,2}$ and Lipschitz sense. Combining with Lemmas \ref{alem3.4} and \ref{lem3.7}, we conclude.

\end{proof}
As a direct consequence of Theorem \ref{GGthm1}, we obtain the following proposition.
\begin{prop}[Integral type Bochner inequality]\label{prop4.5}
	Under Theorem \ref{GGthm1}, assume in addition $h\in \mathrm{Lip}(X,\mathsf{d})$, then for any $t_1<t_2$ we have
	\[
	\begin{aligned}
			&\int_{\Omega_{t_2}\setminus \Omega_{t_1}}\tau\circ f\,\langle h\nabla {|\nabla g|}^2,\nabla f\rangle\,\integral\\
			\geqslant\ & 2\int_{t_1}^{t_2}\tau(t)\int_{\Omega_t}\left(h{|\mathrm{Hess}\, g| }^2+h\langle\nabla \Delta g,\nabla g\rangle+Kh|\nabla g|^2+\mathrm{Hess}\,g(\nabla g,\nabla h)\right)\,\integral\mathrm{d}t.
	\end{aligned}
	\]
\end{prop}
\begin{proof}
	According to the proof of Theorem \ref{GGthm1}, since $g$ is compactly supported, we only consider the case that $X$ is compact and $f\in \mathrm{Test}(X)$.

Given that ${|\nabla g|}^2\in H^{1,2}$, $\mathrm{h}_s({|\nabla g|}^2)$ $H^{1,2}$-converges to ${|\nabla g|}^2$ as $s\to 0$. For any $s>0$, Theorem \ref{GGthm1} implies
	\begin{equation}\label{eqn4.4}
		\int_{\Omega_{t_2}\setminus \Omega_{t_1}}\tau\circ f \,\langle h \nabla \mathrm{h}_s({|\nabla g|}^2),\nabla f\rangle\,\integral=\int_{t_1}^{t_2}\tau(t)\int_{ \Omega_t }\left(h\Delta \mathrm{h}_s({|\nabla g|}^2)+\langle\nabla h,\nabla {\mathrm{h}_s(|\nabla g|}^2)\rangle\right)\,\integral\mathrm{d}t.
	\end{equation}

		Let $\varphi\in C^\infty([0,\infty))$ be such that $0\leqslant\varphi\leqslant 1$, $\varphi\equiv 1$ on $[0,1]$ and $\varphi\equiv 0$ on $(2,\infty)$. Notice that $\mathscr{H}^n(\partial \Omega_t)=0$ for $\mathscr{L}^1$-a.e. $t\in (t_1,t_2)$. For every such $t$ and any $\varepsilon\ll 1$, define $\phi_\varepsilon \in \mathrm{Lip}_{\mathrm{c}}(\Omega_t,\mathsf{d})\cap \mathrm{Test}(X)$ by $\phi_\varepsilon:=1-\varphi(\varepsilon^{-1}(t-f))$. Then for every $s>0$, because $\lim_{\varepsilon\to 0}\|\phi_\varepsilon-\chi_{\Omega_t}\|_{L^2}=0$, it follows that $\lim_{\varepsilon\to 0}\|\mathrm{h}_s(\phi_\varepsilon- \chi_{\Omega_t})\|_{L^\infty}=0$. Hence by \cite[Proposition 5.2.7, Corollary 5.2.9]{G18a} and (\ref{abc2.14}), we obtain
\begin{equation}\label{eqn4.5}
\begin{aligned}
	&\int_{ \Omega_t }h\Delta \mathrm{h}_s({|\nabla g|}^2)\,\integral=\lim_{\varepsilon\to 0}\int_{\Omega_{t}} \phi_\varepsilon\, h\Delta \mathrm{h}_s({|\nabla g|}^2)\,\integral=\lim_{\varepsilon\to 0}\int_{\Omega_{t}} \Delta(\phi_\varepsilon \, h)\,  \mathrm{h}_s({|\nabla g|}^2)\,\integral\\
	=\ & \lim_{\varepsilon\to 0}\int_{X} \Delta (\mathrm{h}_s(\phi_\varepsilon\, h))\, {|\nabla g|}^2\,\integral\geqslant  2\,\lim_{\varepsilon\to 0}\int_{X} \mathrm{h}_s(\phi_\varepsilon\, h)\left({|\mathrm{Hess}\, g| }^2+\langle\nabla\Delta g,\nabla g\rangle+K{|\nabla g|}^2\right)\,\integral\\
	=\ &2\int_{X} \mathrm{h}_s\left(\chi_{\Omega_t}\, h\right)\left({|\mathrm{Hess}\, g| }^2+\langle\nabla\Delta g,\nabla g\rangle+K{|\nabla g|}^2\right)\,\integral.
\end{aligned}
\end{equation}

Since it is clear that
\[
\lim_{s\to 0}\int_{\Omega_{t}}\langle\nabla h,\nabla {\mathrm{h}_s(|\nabla g|}^2)\rangle\,\integral=\int_{\Omega_t}\langle\nabla h,\nabla {|\nabla g|}^2\rangle\,\integral,\ \forall t\in (t_1,t_2),
\]
we conclude by combining (\ref{eqn4.4}), (\ref{eqn4.5}) and letting $s\to 0$. 
\end{proof}
\begin{remark}
	Assume $\tau\equiv 1$. Then for any $t_2,t_3>t_1$ with $t_4-t_3=t_2-t_1$, using change of variables in (\ref{eqn4.4}) gives
	\[
	\begin{aligned}
		&\int_{\Omega_{t_4}\setminus \Omega_{t_3}} \langle h \nabla \mathrm{h}_s({|\nabla g|}^2),\nabla f\rangle\,\integral-	\int_{\Omega_{t_2}\setminus \Omega_{t_1}} \langle h \nabla \mathrm{h}_s({|\nabla g|}^2),\nabla f\rangle\,\integral\\
		=\ &\int_{t_1}^{t_2}\int_{ \Omega_{t_3-t_1+t} \setminus \Omega_t}\left(h\Delta \mathrm{h}_s({|\nabla g|}^2)+\langle\nabla h,\nabla {\mathrm{h}_s(|\nabla g|}^2)\rangle\right)\,\integral\mathrm{d}t.
	\end{aligned}
	\]
	For each $t\in (t_1,t_2)$, we can take the test function $\phi_\varepsilon=1-\varphi(\varepsilon^{-1}(f-t))\varphi(\varepsilon^{-1}(t_3-t_1+t-f))$ in (\ref{eqn4.5}) and let $\varepsilon\to 0$ and then $s\to 0$ to get
	\begin{equation}\label{eqn3.13}
		\begin{aligned}
&			\int_{\Omega_{t_4}\setminus \Omega_{t_3}} \langle h \nabla {|\nabla g|}^2,\nabla f\rangle\,\integral-	\int_{\Omega_{t_2}\setminus \Omega_{t_1}} \langle h \nabla {|\nabla g|}^2,\nabla f\rangle\,\integral\\
\geqslant \ &2\int_{t_1}^{t_2}\int_{\Omega_{t_3-t_1+t}\setminus\Omega_t }\left(h{|\mathrm{Hess}\, g| }^2+h\langle\nabla \Delta g,\nabla g\rangle+Kh|\nabla g|^2+\mathrm{Hess}\,g(\nabla g,\nabla h)\right)\,\integral\mathrm{d}t.
		\end{aligned}
	\end{equation}
\end{remark}

\section{Colding's monotonicity formula on non-smooth setting}\label{sec5}
This section is aimed at generalizing results in \cite[Section 2]{C12}. Assume $(X,\mathsf{d},\mathscr{H}^n)$ is a non-collapsed non-parabolic RCD$(0,n)$ space with $n\geqslant 3$. 
\begin{defn}[Green function]
	The Green function $G$ of $(X,\mathsf{d},\mathscr{H}^n)$ is defined by
	\[
	\begin{aligned}
		G : X \times X \setminus \operatorname{diag}(X)&\longrightarrow (0, \infty)\\
		(x, y) &\longmapsto \int_0^\infty \rho(x, y, t) \, \mathrm{d}t,
	\end{aligned}
	\]
	where $\operatorname{diag}(X) := \{(x, x) \in X \times X \mid x \in X\}$. 
\end{defn}

By (\ref{1.3}) and \cite[Theorem 1.1]{JLZ16}, $G$ is well-defined. In the sequel, we fix a point $x\in X$. It is proved in (the proof of) \cite[Lemma 2.5]{BS20} that $G(x,\cdot)$ is harmonic on $ X \setminus \{x\}$ and that $ G(x,\cdot) \in H_{\mathrm{loc}}^{1,1}(X, \mathsf{d}, \mathscr{H}^n) $ holds with
\[
\int_X \Delta f(y) \, G(x,y) \, \integral(y) = -f(x),\ \forall f\in \mathrm{Test}(X). 
\]

Let
\[
\nu:=\lim_{r\to 0}\frac{\mathscr{H}^n(B_r(x))}{ r^n},\ \ \mathcal{V}:=\lim_{r\to \infty}\frac{\mathscr{H}^n(B_r(x))}{ r^n}.
\]
It can be easily checked from Theorems \ref{BGineq} and \ref{DG18cor1.7} that $0<\nu \leqslant \omega_n$ and $\mathcal{V}\leqslant \nu$.

Define
\[
\mathsf{b}=\mathsf{b}_x:=\left(n(n-2)\,\nu\,G(x,\cdot)\right)^{\frac{1}{2-n}},\ \Omega_t:=\{\mathsf{b}\leqslant t\},
\]

Regarding the properties of $\mathsf{b}$, we summarize from \cite{CM97
	,CM97b,C12,HP25} as follows.

\begin{lem}\label{lem5.2}
	We have 
	\begin{enumerate}
		\item[$(1)$]$($Comparison with distance function$)$ There exists $C=C(n)>1$ such that 
		\begin{align}\label{a5.2}
			C^{-1}F\circ \mathrm{d}_x\leqslant \mathsf{b}\leqslant CF\circ \mathrm{d}_x\ \text{on }X\setminus \{x\},\ \text{where }	F(t):=\left(\int_t^\infty\frac{s}{\mathscr{H}^n(B_s(x))}\,\mathrm{d}s\right)^{\frac{1}{2-n}}.
			\end{align}
		\item[$(2)$]$($Explicit expression of the Laplacian$)$ On $X\setminus\{x\}$ it holds
		\begin{align}\label{5.3}
			\Delta \mathsf{b}^2=2n|\nabla \mathsf{b}|^2,\  \Delta \mathsf{b}=(n-1){|\nabla\mathsf{b}|^2/\mathsf{b}}. 
		\end{align}
		\item[$(3)$]$($Regularity of the gradient$)$ ${|\nabla \mathsf{b}|}^2\in H^{1,2}_{\mathrm{loc}}(X\setminus\{x\},\mathsf{d},\mathscr{H}^n)$ admits an upper semicontinuous representative $($still denoted by ${|\nabla \mathsf{b}|}^2)$ such that 
		\begin{align}\label{5.1}
			|\nabla \mathsf{b}|(z)=\limsup_{y\to z}|\nabla \mathsf{b}|(y)\leqslant 1,\ \forall z\in X\setminus \{x\}.
		\end{align}
		\item[$(4)$]$($Asymptotic formulas$)$
		\begin{align}\label{5.2}
			\lim_{y\to x}\mathsf{b}/\mathsf{d}_x=1,\  \lim_{y\to x} |\nabla \mathsf{b}|(y)=1.
		\end{align}
	\end{enumerate}
\end{lem}

For any $u\in \mathrm{Test}(X)$, let
\[
I_u(t):= t^{1-n}\int_{\partial \Omega_t}u^2|\nabla \mathsf{b}|\,\mathrm{dPer}(\Omega_t,\cdot).
\]
Similarly as in \cite[Section 2]{CM97}, we get the following result.
\begin{lem}\label{lem6.1}
	If $u$ is harmonic on $\Omega_{t_0}$ for some $t_0>0$, then $I_u$ is differentiable on $t\in (0,t_0)$ with
	\[
	I_u'(t)=2\,t^{1-n}\int_{\Omega_t} {|\nabla u|}^2\,\integral.
	\]
	
\end{lem}
\begin{proof}
	We first show
	\begin{equation}\label{6.1a}
		t\,\frac{\mathrm{d}}{\mathrm{d}t}\int_{\Omega_t}u^2{|\nabla \mathsf{b}|}^2\,\mathrm{d}\mathscr{H}^n=	2\int_0^t  s\int_{\Omega_s}{|\nabla u|}^2\,\mathrm{d}\mathscr{H}^n\mathrm{d}s+n\int_{\Omega_t}u^2{|\nabla \mathsf{b}|}^2\,\integral,\ \forall t\in (0,t_0).
	\end{equation}
	%

	Fix $t>0$. For any $\varepsilon\ll t$, take a cut-off function $\varphi_\varepsilon$ such that $0\leqslant \varphi_\varepsilon\leqslant 1$, $\mathrm{supp}\, \varphi_\varepsilon\subset B_{2\varepsilon}(x)$, $\varphi_\varepsilon=1$ on $B_\varepsilon(x)$ and $\varepsilon\|\nabla \varphi_\varepsilon\|_{L^\infty}+\varepsilon^2\|\Delta \varphi_\varepsilon\|_{L^\infty}\leqslant C(n)$. 
	
	Provided that $u$ is harmonic, by Theorem \ref{GGthm}, one directly calculates that
	\begin{equation}\label{feuahfae}
		\begin{aligned}
			\int_{\Omega_t} (1-\varphi_\varepsilon)u\mathsf{b}\langle\nabla u,\nabla \mathsf{b}\rangle\,\mathrm{d}\mathscr{H}^n=\int_0^t s\int_{\Omega_s}\langle\nabla u,\nabla (u(1-\varphi_\varepsilon)) \rangle\,\mathrm{d}\mathscr{H}^n\mathrm{d}s.
		\end{aligned}
	\end{equation}
	Obviously the left hand side of (\ref{feuahfae}) converges to $\int_{\Omega_t} u\mathsf{b}\,\langle\nabla u,\nabla \mathsf{b}\rangle\,\mathrm{d}\mathscr{H}^n$ as $\varepsilon\to 0$. Moreover, since
	\[
	\int_{\Omega_t}|\nabla \varphi_\varepsilon|\,\mathrm{d}\mathscr{H}^n\leqslant \varepsilon^{-1}\mathscr{H}^n(B_{2\varepsilon}(x))\leqslant C(n)\,  \varepsilon^{n-1} \to 0,\ \text{as}\ \varepsilon\to 0,
	\]
	using H\"older inequality we  obtain
	\[
	\left|\int_0^t s\int_{\Omega_s}\langle \nabla u,\nabla \varphi_\varepsilon\rangle\,\mathrm{d}\mathscr{H}^n\right|\leqslant\int_0^t s\int_{\Omega_s}|\langle \nabla u,\nabla \varphi_\varepsilon\rangle|\,\mathrm{d}\mathscr{H}^n\leqslant t^2\int_{\Omega_t} |\nabla u||\nabla \varphi_\varepsilon|\,\mathrm{d}\mathscr{H}^n\to 0,\ \text{as}\ \varepsilon\to 0.
	\]
	Here we may assume  $|u|\leqslant 1$ on $\Omega_{t}$ owing to \cite{J14}. Thus, letting $\varepsilon\to 0$ in (\ref{feuahfae}) we get	
	\begin{equation}\label{eqn6.2}
		\begin{aligned}
			\int_0^t s\int_{\Omega_s}|\nabla u|^2\,\mathrm{d}\mathscr{H}^n\mathrm{d}s=\int_{\Omega_t} u\mathsf{b}\,\langle\nabla u,\nabla \mathsf{b}\rangle\,\mathrm{d}\mathscr{H}^n.
		\end{aligned}
	\end{equation}
On the other hand, by (\ref{5.3}), (\ref{5.1}),
	\[
	\begin{aligned}
		\left|\int_{0}^t\int_{\Omega_s}u\varphi_\varepsilon \Delta \mathsf{b}^2\,\mathrm{d}\mathscr{H}^n\mathrm{d}s\right|\leqslant t\int_{\Omega_t}\varphi_\varepsilon \Delta \mathsf{b}^2\,\mathrm{d}\mathscr{H}^n\leqslant C(n)\,{t \mathscr{H}^n(B_{2\varepsilon}(x))}\to 0\, \ \text{as}\  \varepsilon \to 0.
	\end{aligned}
	\]
	This similarly yields
	\begin{equation}\label{eqn6.3}
		\begin{aligned}
			&\int_{\Omega_t}  u^2\langle\nabla \mathsf{b}^2,\nabla \mathsf{b}\rangle\,\mathrm{d}\mathscr{H}^n\\
			=\ &	\int_0^t \int_{\Omega_s}\left(\langle \nabla u^2,\nabla {\mathsf{b}}^2\rangle+u^2\Delta {\mathsf{b}}^2\right)\,\mathrm{d}\mathscr{H}^n\mathrm{d}s=\int_0^t \int_{\Omega_s}\left(\langle \nabla u^2,\nabla {\mathsf{b}}^2\rangle+2n u^2{|\nabla\mathsf{b}|}^2\right)\,\mathrm{d}\mathscr{H}^n\mathrm{d}s.
		\end{aligned}
	\end{equation}
Combining (\ref{eqn6.2}) and (\ref{eqn6.3}) shows
	\[
	\int_{\Omega_t}u^2\mathsf{b}{|\nabla \mathsf{b}|}^2\,\mathrm{d}\mathscr{H}^n=	2\int_0^t \int_{0}^{s} r\int_{G_r}{|\nabla u|}^2\,\mathrm{d}\mathscr{H}^n\mathrm{d}s\mathrm{d}r+n\int_0^t\int_{\Omega_s}u^2{|\nabla \mathsf{b}|}^2\,\integral\mathrm{d}s.
	\]
	(\ref{6.1a}) then follows by taking derivative with respect to $t$ in the above equation. To conclude, it suffices to use co-area formula and differentiate (\ref{6.1a}).
\end{proof}
Let
\[
A(t):=t^{1-n}\int_{\partial \Omega_t}{|\nabla \mathsf{b}|}^3\,\mathrm{dPer}(\Omega_t,\cdot),\ V(t):=t^{-n}\int_{\Omega_t}{|\nabla \mathsf{b}|}^4\,\integral.
\]
It follows clearly from co-area formula that
\begin{equation}\label{aeqn6.4}
	tV'(t)=A-nV(t),\ \mathscr{L}^1\text{-a.e.}\ t\in (0,\infty).
\end{equation}

\begin{lem}\label{lem6.2}
	$A,V$ are continuous on $[0,\infty)$ with $
	A(0)=n\,\nu,\ V(0)=\nu$. Moreover, $A\in H^{1,1}_{\mathrm{loc}}((0,\infty))$ and $V\in C^1((0,\infty))$.
\end{lem}
\begin{proof}

	To see the continuity of $A$, firstly by Theorem \ref{GGthm} and (\ref{11eqn2.16}), we have
	\[
	\begin{aligned}
		&\int_{\Omega_t}{|\nabla \mathsf{b}|}^4\,\integral=\int_0^t \int_{\Omega_s}\left(|\nabla \mathsf{b}|^2\Delta \mathsf{b}+\langle\nabla |\nabla \mathsf{b}|^2,\nabla \mathsf{b}\rangle\right)\,\integral\mathrm{d}s\\
		=\ &\int_0^t \int_{\Omega_s}\left(\frac{(n-1){|\nabla\mathsf{b}|}^4}{\mathsf{b}}+2\,\mathrm{Hess}\,\mathsf{b}(\nabla\mathsf{b},\nabla\mathsf{b})\right)\,\integral\mathrm{d}s.
	\end{aligned}
	\]
	This together with the co-area formula shows
	\begin{equation}\label{eqn6.5}
		\int_{\partial \Omega_t}{|\nabla \mathsf{b}|}^3\,\mathrm{dPer}(\Omega_t,\cdot)=\int_{\Omega_t}\left(\frac{(n-1){|\nabla\mathsf{b}|}^4}{\mathsf{b}}+2\,\mathrm{Hess}\,\mathsf{b}(\nabla\mathsf{b},\nabla\mathsf{b})\right)\,\integral.
	\end{equation}
	
	Because $s\mapsto s^{-n}\mathscr{H}^n(B_s(x))$ is non-increasing, differentiating it and using the coarea formula along with Theorem \ref{BGineq} yields, $s\mathscr{H}^{n-1}(\partial B_s(x))\leqslant n\mathscr{H}^n(B_s(x))$ for $\mathscr{L}^1$-a.e. $s\in (0,\infty)$. Thus for any $m\in [0,4)$, Theorem \ref{BGineq} and (\ref{5.2}) yield
	\[
	\lim_{t\to 0} \int_{ \Omega_t }\mathsf{b}^{-m}\,\integral\leqslant \lim_{t\to 0} \int_{ B_{2t}(x) }{\mathsf{d}_x}^{-m}\,\integral\leqslant \lim_{t\to 0} \int_{0}^t \frac{\mathscr{H}^{n-1}(\partial B_s(x))}{s^m}\,\mathrm{d}s\leqslant \lim_{t\to 0} \int_{0}^t s^{n-m}\,\mathrm{d}s=0.
	\]
	Upon choosing $R\gg 1$ with $\Omega_t\subset B_R(x)$, we know from Theorem \ref{BGineq} and (\ref{eeeeeeqn2.6}) that
	\begin{equation}\label{hfieoahfoaeifo}
	\begin{aligned}
		&\int_{\Omega_t}\left|\mathrm{Hess}\,\mathsf{b}(\nabla\mathsf{b},\nabla\mathsf{b})\right|\,\integral\leqslant 	\int_{B_R(x)}{\left|\mathrm{Hess}\,\mathsf{b}\right| }^2\,\integral\\
		\leqslant\ & C(n,R)\int_{B_{2R}(x)}{(\Delta \mathsf{b})^2}\,\integral\leqslant C(n,R)\int_{B_{2R}(x)}\mathsf{b}^{-2}\,\integral<\infty.
	\end{aligned}
	\end{equation}

	Thus by understanding $\langle\nabla{|\nabla \mathsf{b}|}^2,\nabla \mathsf{b}\rangle /|\nabla \mathsf{b}|$ as $0$ at $\{|\nabla \mathsf{b}|=0\}$, one can use co-area formula and (\ref{hfieoahfoaeifo}) to see that
	\[
	t\mapsto \int_{\partial \Omega_t} \frac{1}{|\nabla \mathsf{b}|}\langle\nabla{|\nabla \mathsf{b}|}^2,\nabla \mathsf{b}\rangle \,\mathrm{dPer}(\Omega_t,\cdot)\in \, L^{1}_{\mathrm{loc}}((0,\infty)),
	\]
	which proves $A\in H^{1,1}_{\mathrm{loc}}(0,\infty)$.

Regarding the limit $\lim_{t\to 0}A(t)$, by (\ref{5.2}),  it suffices to consider the limit $\lim_{t\to 0} I_1(t)$, where 
	\[
	I_1(t):= t^{1-n}\int_{\partial \Omega_t}{|\nabla \mathsf{b}|}\,\mathrm{dPer}(\Omega_t,\cdot).
	\]
Lemma \ref{lem6.1} verifies that $I_1$ is a constant function. Thus using co-area formula again, we see
	\[
	\int_{\Omega_t}{|\nabla\mathsf{b}|}^2\,\integral=\int_0^t s^{n-1}I_1(s)\,\mathrm{d}s=\frac{t^n I_1(1)}{n} .
	\]
	Combined with (\ref{5.2}), this gives
	\begin{align}\label{0421a}
		\lim_{t\to 0} A(t)=\lim_{t\to 0}I_1(t)=I_1(1)=n\,\lim_{t\to 0}\frac{\int_{\Omega_t}{|\nabla \mathsf{b}|}^2\,\integral}{t^n}=n\,\nu.
	\end{align}
	
	The regularity of $V$ and the limit $V(0)=\nu$ follow in the same manner.
\end{proof}

\begin{thm}[The differentiability of $A$]\label{thm6.3}
	$A$ is locally Lipschitz on $(0,\infty)$ with $A'\leqslant 0$ $\mathscr{L}^1$-a.e. 
\end{thm}
\begin{proof}	By (\ref{11eqn2.16}), (\ref{eqn6.5}) and co-area formula,
		\[
	\begin{aligned}
		t^{n-1}	A(t)&=(n-1)\int_0^t s^{n-2}A(s)\,\mathrm{d}s+\int_{\Omega_t}\langle\nabla{|\nabla \mathsf{b}|}^2,\nabla\mathsf{b}\rangle\,\integral\\
		&=(n-3)\int_0^t s^{n-2}A(s)\,\mathrm{d}s+\int_{\Omega_t}\frac{1}{4\mathsf{b}^2}\langle\nabla{|\nabla \mathsf{b}^2|}^2,\nabla\mathsf{b}\rangle\,\integral.
	\end{aligned}
	\]
	
	Set 
	\[
	E:=\left\{t\in (0,\infty):\ \lim_{s\to t}\frac{A(s)-A(t)}{s-t}=A'(t)\right\},
	\]
	which has full measure in $(0,\infty)$.

	If we plug $g={|\nabla \mathsf{b}^2|}^2$, $h=1$ and $\tau=t^{-2}$ into Proposition \ref{prop4.5}, and take derivative with respect to $t$, we get
	\begin{equation}\label{eqn6.6}
		\begin{aligned}
			t^{n+1}A'(t)+2t^n A(t)\geqslant \frac{1}{2}\int_{\Omega_t}\left({|\mathrm{Hess}\,\mathsf{b}^2| }^2+\langle\nabla \Delta \mathsf{b}^2,\nabla\mathsf{b}^2\rangle\right)\,\integral,\ \forall t\in E.
		\end{aligned}
	\end{equation}
	Since Theorem \ref{GGthm} implies 
	\begin{equation}\label{eqn6.7}
		\int_{\Omega_t}\Delta \mathsf{b}^2\langle\nabla\mathsf{b}^2,\nabla\mathsf{b}\rangle\,\integral=\int_0^t \int_{\Omega_s}\left({(\Delta \mathsf{b}^2)}^2+\langle\nabla\Delta\mathsf{b}^2,\nabla \mathsf{b}^2\rangle\right)\,\integral\mathrm{d}s,\ \forall t>0,
	\end{equation}
taking derivative with respect to $t$ in (\ref{eqn6.7}) and applying (\ref{eqn6.6}) give
	\begin{equation}\label{eqn6.8}
		t^{n+1}A'(t)+2t^n A(t)\geqslant \frac{1}{2}\int_{\Omega_t}\left({|\mathrm{Hess}\,\mathsf{b}^2| }^2-{(\Delta\mathsf{b}^2)}^2\right)\,\integral+\int_{\partial \Omega_t} \mathsf{b}\Delta \mathsf{b}^2|\nabla\mathsf{b}|\,\mathrm{dPer}(\Omega_t,\cdot).
	\end{equation}
	
	Let $\mathrm{g}$ be the standard Riemannian metric on $X$. Then it follows from Proposition \ref{prop1.1} that
	\begin{align}\label{abcde1}
	\int_{\Omega_t}{\left|\mathrm{Hess}\,\mathsf{b}^2-\frac{\Delta \mathsf{b}^2}{n} \mathrm{g}\right| }^2\,\integral=\int_{\Omega_t}{|\mathrm{Hess}\,\mathsf{b}^2| }^2\,\integral-\int_{\Omega_t}\frac{{(\Delta\mathsf{b}^2)}^2}{n}\,\integral,\ \forall t>0.
	\end{align}
	This together with (\ref{5.3}) and (\ref{eqn6.8}) yields
	\begin{align}\label{0421c}
	t^{n+1}A'(t)+2(1-n) t^n A(t)+2n(n-1)t^n V(t)\geqslant \int_{\Omega_t}{\left|\mathrm{Hess}\,\mathsf{b}^2-\frac{\Delta\mathsf{b}^2}{n} \mathrm{g}\right| }^2\,\integral\geqslant 0.
	\end{align}

	On the other hand, for every $t_4>t_3>t_1$ with $t_4-t_3=t_2-t_1$, let us take a cut-off function $\varphi$ on $(0,\infty)$ such that $\varphi=0$ on $(0,t_1/4)$ and $\varphi=1$ on $(t_1/2,\infty)$. By letting $\tilde{\mathsf{b}}:=\varphi\circ\mathsf{b}$, plugging $g=|\nabla \mathsf{b}|^2$, $h=\tilde{\mathsf{b}}^{4-2n}$ and $\tau=t^{4-2n}$ into Proposition \ref{eqn3.13} and calculating as above we obtain
	\begin{equation}\label{eqn5.13}
		\int_{\Omega_{t_4}\setminus \Omega_{t_3}}\mathsf{b}^{4-2n}\langle\nabla {|\nabla \mathsf{b}|}^2,\nabla \mathsf{b}
		\rangle\,\integral-	\int_{\Omega_{t_2}\setminus \Omega_{t_1}}\mathsf{b}^{4-2n}\langle\nabla {|\nabla \mathsf{b}|}^2,\nabla \mathsf{b}
		\rangle\,\integral\geqslant 0.
	\end{equation}
From Theorem \ref{GGthm} we know for every $t>t_1/2$ it holds
	\[
	\begin{aligned}
		&t^{3-n} A(t)=\int_{\partial \Omega_t} {\tilde{\mathsf{b}}}^{4-2n}{|\nabla \mathsf{b}|}^3\,\mathrm{dPer}(\Omega_t,\cdot)\\
		=\ &\int_{\Omega_t}{\tilde{\mathsf{b}}}^{4-2n}|\nabla \mathsf{b}|^2\Delta\mathsf{b}\,\integral+(4-2n) \int_{\Omega_t} {\tilde{\mathsf{b}}}^{3-2n}|\nabla \mathsf{b}|^2\langle\nabla \tilde{\mathsf{b}},\nabla\mathsf{b}\rangle\,\integral+\int_{\Omega_{t}}{\tilde{\mathsf{b}}}^{4-2n}\langle\nabla {|\nabla \mathsf{b}|}^2,\nabla \mathsf{b}\rangle\,\integral,
	\end{aligned}
	\]
	which combined with (\ref{5.3}), (\ref{eqn5.13}) and co-area formula shows
	\begin{equation}\label{eqn5.14}
		\begin{aligned}
			t^{3-n}A{\big|}_{t_3}^{t_4}-t^{3-n}A{\big|}_{t_1}^{t_2}
			\geqslant (3-n)\left(\int_{t_3}^{t_4} t^{2-n} A(t)\,\mathrm{d}t-\int_{t_1}^{t_2} t^{2-n} A(t)\,\mathrm{d}t.\right).
		\end{aligned}
	\end{equation}
For every $t_1,t_3\in E$ with $t_1<t_3$, multiplying (\ref{eqn5.14}) with $1/(t_2-t_1)$ and then letting $t_2\to t_1$ yields 
	\begin{align}\label{0421b}
	{t_3}^{3-n} A'(t_3)-{t_1}^{3-n}A'(t_1)\geqslant 0.
	\end{align}
According to (\ref{6.1a}), (\ref{0421a}) and the fact that $I_1$ is a constant, $A$ is bounded above by $n \nu$. By (\ref{0421b}), $A'(t)$ must be non-positive for any $t\in E$.
    Hence (\ref{0421c}) and the continuity of $A$ and local Lipschitz regularity of $V$ prove the local Lipschitz continuity of $A$.
\end{proof}
\begin{remark}
	In the proof of (\ref{eqn5.13}) we indeed implicitly show $\mathcal{L}({|\nabla \mathsf{b}|}^2)\geqslant 0$ in the integral type Gauss-Green sense. Here $\mathcal{L}({|\nabla \mathsf{b}|}^2)$ can be interpreted as $\Delta {|\nabla \mathsf{b}|}^2+2\langle \nabla G(x,\cdot),\nabla {|\nabla \mathsf{b}|}^2\rangle$ in the measured Laplacian sense. Compare \cite[Proposition 3.21]{HP25}.
\end{remark}
\begin{cor}\label{cor4.6}
 $V$ is non-increasing and $A-2(n-1)V$ is non-decreasing on $[0,\infty)$. Moreover, $A\leqslant nV$.
\end{cor}
\begin{proof}
	It follows directly from (\ref{aeqn6.4}) and (\ref{0421c}) that
	\begin{align}\label{0421d}
	(A-2(n-1)V)'(t)\geqslant \frac{1}{2t^{n+1}} \int_{\Omega_t}{\left|\mathrm{Hess}\,\mathsf{b}^2-\frac{\Delta\mathsf{b}^2}{n} \mathrm{g}\right| }^2\,\integral.
	\end{align}
By Theorem \ref{thm6.3}, the above inequality also shows $2(n-1)V'\leqslant A'\leqslant 0$.  Applying (\ref{aeqn6.4}), we conclude.
\end{proof}

Let 
\[
W(t):=\int_1^t s^{-n}\int_{ \partial \Omega_s} ({|\nabla \mathsf{b}|}^3-|\nabla \mathsf{b}|)\,\mathrm{dPer}(\Omega_s,\cdot)\mathrm{d}s. 
\]
\begin{thm}\label{thm5.4}
	$W$ is non-increasing and $A-2(n-2)W$ is non-decreasing on $(0,\infty)$.
\end{thm}
\begin{proof}
	Notice that
	\begin{align}\label{4.22}
	W'(t)=t^{-n}\int_{ \partial \Omega_t} ({|\nabla \mathsf{b}|}^3-|\nabla \mathsf{b}|)\,\mathrm{dPer}(\Omega_t,\cdot)=\frac{A(t)-n\nu }{t}.
	\end{align}
	The conclusion then follows from an argument similar to the proofs of Theorem \ref{thm6.3} and of \cite[Lemma 2.7, Theorems 2.8 and 2.9]{C12}. In particular, we have
	\begin{align}\label{0421e}
	\left(t^{2-n}(A-n\nu)\right)'(t)\geqslant\frac{1}{2t^{n-1}}\int_{\Omega_t}\frac{1}{\mathsf{b}^n}{\left|\mathrm{Hess}\,\mathsf{b}^2-\frac{\Delta\mathsf{b}^2}{n} \mathrm{g}\right| }^2\,\integral;
	\end{align}
	\begin{align}\label{0421f}
	\left(A-(n-2)W\right)'(t)\geqslant\frac{1}{2t}\int_{\Omega_t}\frac{1}{\mathsf{b}^n}{\left|\mathrm{Hess}\,\mathsf{b}^2-\frac{\Delta\mathsf{b}^2}{n} \mathrm{g}\right| }^2\,\integral.
	\end{align}
\end{proof}

Before we give the rigidity result, let us first recall the definition of metric cone.
\begin{defn}[\cite{K15}]
	An RCD$(0,N)$ space $({X},\mathsf{d},\mathfrak{m})$ is said to be a metric cone with base point $o^\ast$ if there exists an RCD$(N-2,N-1)$ space $(Z,\mathsf{d}_Z,\mathfrak{m}_Z)$ with $\mathrm{diam}(Z)\leqslant \pi$ such that the following holds. 
	\begin{enumerate}
		
		\item[$(1)$] $X=[0,\infty)\times Z/(\{0\}\times Z)$ with the origin denoted by $o^\ast$, and the measure satisfies \[\mathrm{d}\mathfrak{m}(r,x)=r^{N-1}\mathrm{d}r\otimes \mathrm{d}\mathfrak{m}_Z(x).
		\]
		\item[$(2)$] For any two points $(r,x)$ and $(s,y)$, the distance between them is defined as
		\[
		\mathsf{d}\left((r,x),(s,y)\right):=\sqrt{r^2+s^2-2rs \cos\left(\mathsf{d}_Z(x,y) \right)}.
		\]
	\end{enumerate}
\end{defn}

\begin{thm}[{\cite[Theorem 1.6]{HP25}}]\label{thm4.10}
	If there exists a point $y\in X\setminus \{x\}$ such that $|\nabla \mathsf{b}|(y)=1$, then $\mathsf{b}=\mathsf{d}_x$ and $X$ is a metric cone with base point $x$.
\end{thm}

\begin{thm}[Rigidity I]\label{thm4.11}
	If for some $t_1,t_2\in (0,\infty)$ with $t_1<t_2$, either of the following holds 
	\[
	A(t_2)-2(n-1)V(t_2)=A(t_1)-2(n-1)V(t_1),\  \exists 0\leqslant t_1<t_2\leqslant \infty;
	\]
	\[
	{t_2}^{2-n}(A(t_2)-n\nu)={t_1}^{2-n}(A(t_1)-n\nu),\ \exists 0<t_1<t_2<\infty;\]
	\[ 
	A(t_2)-2(n-2)W(t_2)=A(t_1)-2(n-2)W(t_1),\ \exists 0<t_1<t_2<\infty,
	\]then $X$ is a metric cone with base point $x$.
\end{thm}
\begin{proof}
	In either of the cases, by (\ref{0421d}) (\ref{0421e}) and (\ref{0421f})  we have 
	\[
	\mathrm{Hess}\,\mathsf{b}^2=\frac{\Delta \mathsf{b}^2}{n}\mathrm{g}=2{|\nabla \mathsf{b}|}^2\mathrm{g},\ \text{on}\ \Omega_{t_2}.
	\]
	This together with (\ref{11eqn2.16}) implies on $\Omega_{t_2}\setminus \{x\}$, 
	\begin{equation}\label{feafe}
		\begin{aligned}
			&{\left|\nabla {|\nabla \mathsf{b}|}^2\right|}^2=\langle\nabla {|\nabla \mathsf{b}|}^2,\nabla {|\nabla \mathsf{b}|}^2\rangle=2\mathrm{Hess}\,\mathsf{b}(\nabla\mathsf{b},\nabla {|\nabla \mathsf{b}|}^2)\\
			=\ &\frac{1}{\mathsf{b}}\left(\mathrm{Hess}\, \mathsf{b}^2(\nabla\mathsf{b},\nabla {|\nabla \mathsf{b}|}^2)-2\,d\,\mathsf{b}\otimes d\,\mathsf{b}(\nabla\mathsf{b},\nabla {|\nabla \mathsf{b}|}^2)\right)=0\  \mathscr{H}^n\text{-a.e.}
		\end{aligned}
	\end{equation}
	By (\ref{5.1}) and (\ref{5.2}), $|\nabla\mathsf{b}|\equiv 1$ on $\Omega_{t_2}$. The conclusion then follows from Theorem \ref{thm4.10}. 
\end{proof}
\begin{thm}[Rigidity II]\label{thm4.12}
		If for some $t_1,t_2\in (0,\infty)$ with $t_1<t_2$, either of the following holds 
		\[
		A(t_1)=A(t_2),\ V(t_1)=V(t_2),\ W(t_1)=W(t_2),
		\]
		then $X$ is a metric cone with base point $x$.
\end{thm}
\begin{proof}
	In either of the cases, by Corollary \ref{cor4.6}, Theorems \ref{thm6.3} and \ref{thm5.4}, we have $A'\equiv 0$ on $[0,t_2)$. Therefore $A\equiv n\nu$ on $[0,t_2)$. Notice that 
	\[
	\int_0^t s^{n-1}A(s)\,\mathrm{d}s=\int_{\Omega_t} {|\nabla \mathsf{b}|}^4\,\integral\leqslant \int_{\Omega_t} {|\nabla \mathsf{b}|}^2\,\integral=\int_0^t s^{n-1}I_1(s)\,\mathrm{d}s,\ \forall t\in (0,t_2).
	\] Since now $A\equiv n\nu=I_1(t)$, $|\nabla \mathsf{b}|\equiv 1$ on $\Omega_{t_2}$. Hence applying Theorem \ref{thm4.10} we conclude.
\end{proof}
\begin{prop}[Asymptotic behavior at infinity]\label{prop4.13}
	We have
	\begin{align}\label{4.25}
	\lim_{t\to \infty} A=n\nu \left(\frac{\nu}{\mathcal{V}}\right)^{\frac{2}{2-n}},\ 	\lim_{t\to \infty} V=\nu \left(\frac{\nu}{\mathcal{V}}\right)^{\frac{2}{2-n}}.
	\end{align}
\end{prop}
\begin{proof}
	If $\mathcal{V}=0$, then by Lemma \ref{lem5.2} and Theorem \ref{BGineq}, for any $t_0>0$ and any $y\notin \Omega_{t_0}$ we have
	\[
	\begin{aligned}
			&G(x,y)\geqslant C\int_{\mathsf{d}(x,y)}^\infty \frac{s}{\mathscr{H}^n(B_s(x))}\,\mathrm{d}s=\frac{C}{\mathscr{H}^n(B_{t_0}(x))}\int_{\mathsf{d}(x,y)}^\infty \frac{s\mathscr{H}^n(B_{t_0}(x))}{\mathscr{H}^n(B_s(x))}\,\mathrm{d}s\\
			\geqslant\  &\frac{C}{\mathscr{H}^n(B_{t_0}(x))}\int_{\mathsf{d}(x,y)}^{\infty}s^{1-n}\mathrm{d}s= \frac{C{t_0}^n}{\mathscr{H}^n(B_{t_0}(x))\mathsf{d}^{n-2}(x,y)}.
	\end{aligned}
	\]
	Since it follows from \cite[Theorem 1.2]{J14} that 
	\[
	|\nabla G(x,\cdot)|(y)\leqslant \frac{CG(x,y)}{\mathsf{d}(x,y)},
	\]
	we see
	\[
	|\nabla \mathsf{b}|(y)=\frac{1}{n-2}G^{\frac{1}{2-n}-1}|\nabla G(x,\cdot)|(y)\leqslant C \left(\frac{\mathscr{H}^n(B_{t_0}(x))}{{t_0}^n}\right)^{\frac{1}{n-2}}.
	\]
	Therefore $\lim_{t\to \infty} A=\lim_{t\to \infty}V=0$.

	Regarding the case $\mathcal{V}>0$, assume $(X_i,\mathsf{d}_i,\mathscr{H}^n,x):=(X,{t_i}^{-1}\mathsf{d},\mathscr{H}^n,x)$ pmGH converges to an RCD$(0,n)$ space $(X_\infty,\mathsf{d}_\infty,\mathfrak{m}_\infty,x_\infty)$ for some $t_i \uparrow \infty$. According to \cite[Theorem 1.3]{DG18}, $\mathfrak{m}_\infty=\mathscr{H}^n$. 
	
	Let $G_i$ be the Green function of $X_i$ and $\mathsf{b}_i:=(n(n-2)\nu \, G_i(x,\cdot))^{1/(2-n)}$. By \cite{HP25}, $\mathsf{b}_i=t_i\mathsf{b}$, $|\nabla_i \mathsf{b}_i|=|\nabla_X\mathsf{b}|$ and the Green function of $X_\infty$ satisfies 
\[
G_\infty(x_\infty,\cdot)=\frac{1}{n(n-2)\mathcal{V}}\,{\mathsf{d}_\infty}^{2-n}(x_\infty,\cdot).
\]
This implies $\mathsf{b}_\infty=(\nu/\mathcal{V})^{1/(2-n)}$. Thus
\[
\lim_{t\to \infty}V=\lim_{i\to\infty } \frac{1}{{t_i}^n}\int_{\Omega_{t_i}} {|\nabla \mathsf{b}|}^4\,\integral=\int_{\Omega^\infty_1}{|\nabla^\infty \mathsf{b}_\infty |}^4\,\integral=
\int_{\{\mathsf{d}_\infty\leqslant (\nu/\mathcal{V})^{1/(n-2)}\}}\left(\frac{\nu}{\mathcal{V}}\right)^{\frac{4}{2-n}}\,\integral=\nu \left(\frac{\nu}{\mathcal{V}}\right)^{\frac{2}{2-n}}.
\]
Recall that $A$ is non-increasing and non-negative. Hence from co-area formula we obtain
\[
\frac{2^n-1}{n}\lim_{i\rightarrow \infty} A=\lim_{i\to \infty}\frac{1}{{t_i}^n}\int_{t_i}^{2t_i}s^{n-1}A(s)\,\mathrm{d}s=\lim_{i\to \infty}\frac{1}{{t_i}^n}\int_{\Omega_{2t_i}\setminus \Omega_{t_i}}{|\nabla\mathsf{b}|}^4\,\integral=(2^n-1)\lim_{t\to \infty}V.
\]
\end{proof}
\begin{thm}[Rigidity III]\label{thm4.14}
	If $\inf W>-\infty$, then $X$ is a metric cone.
\end{thm}
\begin{proof}
Assume $\inf W>-\infty$. Since $A$ is non-increasing, by (\ref{4.22}) and Proposition \ref{prop4.13},
\[
0=\lim_{t\to\infty}(W(2t)-W(t))=\lim_{t\to\infty}\int_t^{2t} W'(s)\mathrm{d}s\leqslant \log 2 \lim_{t\to\infty}(A(t)-n\,\nu). 
\]
Thus $\lim_{t\to \infty} A=\lim_{t\to 0} A$. In particular, $A'\equiv 0$. The conclusion then follows from Theorem \ref{thm4.12}.

\end{proof}
\begin{remark}
	If $X$ is a metric cone with base point $x$ such that
	\[
	\left(\mathbb{R}^n,\mathsf{d}_{\mathbb{R}^n},{\omega_n}^{-1}\mathscr{L}^n,0_n\right)\in \mathrm{Tan}(X,\mathsf{d},\mathscr{H}^n,x),
	\] then since the blow up of a metric cone at the base point is itself, it is clear that $X=\mathbb{R}^n$.
\end{remark}
\section{Asymptotic formula for level sets}\label{sec4}
In this section, we prove Theorem \ref{mcthm}. We fix the function $f$ as stated in Theorem \ref{mcthm}.

\begin{defn}[Almost harmonic splitting map]\label{defn1.3}
	Let $(X,\mathsf{d},\mathfrak{m})$ be an RCD$(-1, N)$ space. Let $x \in X$ and $\delta > 0$ be given. A map $\mathbf{U} = (u_1, \ldots, u_k): B_r(x) \rightarrow \mathbb{R}^k$ is said to be a $(k,\delta)$-almost harmonic splitting  map provided:
	\begin{enumerate}
		\item[$(1)$] There exists $C>0$ such that $u_a: B_r(x) \rightarrow \mathbb{R}$ is $C$-Lipschitz continuous for every $a = 1, \ldots, k$.
		\item[$(2)$] $\|\Delta u_a\|_{L^\infty(B_r(x))}\leqslant \delta$ for every $a = 1, \ldots, k$, 
		\item[$(3)$]$r^2 \fint_{B_r(x)} \left| \mathrm{Hess}\, u_a  \right|^2  \mathop{\mathrm{d}\mathfrak{m}} \leqslant \delta$ for every $a = 1, \ldots, k$,
		\item[$(4)$] $\fint_{B_r(x)} \left| \langle\nabla u_a,\nabla u_b\rangle - \delta_{ab} \right|  \mathop{\mathrm{d}\mathfrak{m}} \leqslant \delta$ for every $a, b = 1, \ldots, k$.
	\end{enumerate}
\end{defn}
{\begin{remark}
		If moreover $\mathbf{U}$ is harmonic on $B_r(x)$, then it is said to be a $(k,\delta)$-splitting map. Such maps played essential roles in \cite{ChCo1,ChCo3}, and recently \cite{BPS21,BPS23b,BPS23}.
	\end{remark}}

\begin{prop}\label{prop4.1}
Let $x_0\in \mathcal{R}_n$ and assume $f\in D(\Delta,B_\varepsilon(x_0))$ for some $\varepsilon>0$. Additionally, assume $|\nabla f|(x_0)>0$ and $x_0$ is both a Lebesgue point of ${|\nabla f|}^2$ and $(\Delta f)^2$. Then for any $\delta>0$, there exists $\varepsilon_0\in (0,\varepsilon)$ such that for any $t\in (0,\varepsilon_0]$ one can find $(n-1)$ harmonic functions $u_1,\ldots,u_{n-1}$ such that the following map is an $(n,\delta)$-almost harmonic splitting map. 
\[
\begin{aligned}
\mathbf{U}:B_{t}(x_0)&\longrightarrow \mathbb{R}^n\\
                  x&\longmapsto \left(u_1,\ldots,u_{n-1},|\nabla f|^{-1}(x_0)\mathop{(f-f(x_0))}\right).
\end{aligned}
\]
\end{prop}
\begin{proof}
We proceed by contradiction. Suppose there exists a sequence $r_i\rightarrow 0$ such that there is no desired almost harmonic splitting map on any $B_{r_i}(x_0)$. Since $x_0\in \mathcal{R}_n$, we have
\[
(X_i,\mathsf{d}_i,\mathscr{H}^n,x_0):=\left(X,{r_i}^{-1}\mathsf{d},{r_i}^{-n}\mathscr{H}^n_\mathsf{d},x_0\right)\xrightarrow{\mathrm{pmGH}}(\mathbb{R}^n,\mathsf{d}_{\mathbb{R}^n},\mathscr{L}^n,0_n).
\]
By Proposition \ref{prop:blow.up.bx}, the sequence $\{f_i:=(f-f(x_0))/(r_i|\nabla f|(x_0))\}$ $H_{\mathrm{loc}}^{1,2}$-strongly converges to a linear function $u_\infty^n$ on $\mathbb{R}^n$ with $|\nabla_{\mathbb{R}^n} u_\infty^n|\equiv 1$. Choose linear functions $u_\infty^1,\ldots,u_\infty^{n-1}$ on $\mathbb{R}^n$ such that 
\[\langle \nabla_{\mathbb{R}^n}u^a_\infty, \nabla_{\mathbb{R}^n}u^b_\infty \rangle=\delta_{ab},\ a,b=1,\ldots,n.
\]Proposition \ref{AH182} ensures the existence of harmonic functions $u^1_i,\ldots,u^{n-1}_i$ on $B^{X_i}_3(x_0)$ such that $\{u^a_i\}$ $H^{1,2}$-strongly converges to $u^a_\infty$ on $B^{X_i}_2(x_0)$ for each $a$. It follows from \cite[Theorem 1.1]{J14} that each $u_i^a$ is $C(n)$-Lipschitz continuous on $B_2^{X_i}(x_0)$. Consequently, the map \[\mathbf{U}_i=\left(r_iu^1_i,\ldots,r_iu^{n-1}_i,(f-f(x_0))/|\nabla f|(x_0)\right)
\] meets the condition (1) in Definition \ref{defn1.3}. Furthermore, we have
\[
\lim_{i\rightarrow \infty} \fint_{B_{r_i}(x_0)}|\langle \nabla (r_i u_i^a) ,\nabla (r_i u_i^b)\rangle-\delta_{ab}|\mathop{\mathrm{d}\mathscr{H}^n}=\lim_{i\rightarrow \infty} \fint_{B_1^i(x_0)}|\langle \nabla_i u_i^a,\nabla_i u_i^b\rangle-\delta_{ab}|\mathop{\mathrm{d}\mathscr{H}^n}=0,
\]
and
\[
\lim_{i\rightarrow \infty} \fint_{B_{r_i}(x_0)}|\langle \nabla (r_i u_i^a) ,\nabla f\rangle|\mathop{\mathrm{d}\mathscr{H}^n}=\lim_{i\rightarrow \infty} \fint_{B_1^i(x_0)}|\langle \nabla_i u_i^a,\nabla_i f_i\rangle|\mathop{\mathrm{d}\mathscr{H}^n}=0.
\]

Let us now verify condition (4). Owing to Theorem \ref{BGineq}, (\ref{eeeeeeqn2.6}) and the $H^{1,2}$-strong convergence of $\{u_i^a\}$ on $B_2(0_n)$, for each $a$ it holds that
\[
\begin{aligned}
&{r_i}^2\fint_{B_{r_i}(x_0)}{|\mathrm{Hess}\,(r_i u^a_i)|}^2\mathop{\mathrm{d}\mathscr{H}^n}\\
\leqslant\ & C(K,n)\left(\fint_{B_2^{X_i}(x_0)}\left|{|\nabla_i u_i^a|}^2-1\right|
\mathop{\mathrm{d}\mathscr{H}^n}
+{r_i}\fint_{B_2^{X_i}(x_0)}{|\nabla_i u^a_i|}^2\mathop{\mathrm{d}\mathscr{H}^n}\right)\rightarrow 0 \ \ \mathrm{as }\ i\rightarrow \infty.
\end{aligned} 
\]
As for $f$, since $x_0$ is a Lebesgue point of ${|\nabla f|}^2$ and ${(\Delta f)}^2$, Theorem \ref{BGineq} together with (\ref{eeeeeeqn2.6}) implies
\[
\lim_{i\rightarrow \infty} {r_i}^2 \fint_{B_{r_i}(x_0)} {|\mathrm{Hess}\,f|}^2\mathop{\mathrm{d}\mathscr{H}^n}=0.
\]
Therefore we deduce the contradiction.
\end{proof}

Let us now fix a sufficiently small $\delta>0$. For every $x_0\in \mathcal{R}_n$  that is both a Lebesgue point of ${|\nabla f|}^2$ and ${(\Delta f)}^2$ with $|\nabla f|(x_0)>0$, Proposition \ref{prop4.1} verifies the existence of $\varepsilon_0>0$ and an $(n,\delta)$-almost harmonic splitting map $\mathbf{U}$ on $B_{\varepsilon_0}(x_0)$ such that $u_n:=|\nabla f|^{-1}(x_0)\mathop{(f-f(x_0))}$. Recall in \cite[Section 4]{BMS23}, there is a subset $E\subset B_{\varepsilon_0}(x_0)$ satisfying the following properties:
\begin{itemize}
 \item $E\subset \mathcal{R}_n$ and $\mathscr{H}^n(B_{\varepsilon_0}(x_0)\setminus E)\leqslant C\,\delta\,\mathscr{H}^n(B_{\varepsilon_0}(x_0))$, 
\item For every $x\in E$, we have $|\langle \nabla u_i,\nabla u_j \rangle(x)-\delta_{ij}|\leqslant \sqrt{\delta}$. Moreover, there exist $\varepsilon_x\in (0,(\varepsilon_0-\mathsf{d}(x,x_0))/2)$ and constants $\{a_{ik}\}_{i,k=1,\ldots,n}$ and $\{b^i_{jk}\}_{i,j,k=1,\ldots,n}$ such that the map 
\[
\mathbf{V}=(v_1,\ldots,v_n):B_{\varepsilon_x}(x)\to \mathbb{R}^n
\] defined by $v_i=\sum_j a_{ij} (u_j-u_j(x))+\sum_{j,k,\alpha,\beta} b^i_{jk}a_{j\alpha} a_{k\beta} (u_\alpha-u_\alpha(x))(u_\beta-u_\beta(x))$ satisfies the following properties:

\begin{enumerate}
\item[$(1)$]For any $i, j = 1, \ldots, n$ it holds
    \begin{equation}\label{eqn3.2a}
    \langle\nabla v_i , \nabla v_j\rangle(x) = \delta_{ij}, \quad \lim_{t \to 0} \fint_{B_t(x)} \left| \langle \nabla v_i,  \nabla v_j \rangle- \delta_{ij} \right| \mathop{\mathrm{d}\mathscr{H}^n} = 0.
    \end{equation}
    \item[$(2)$] For any $k = 1, \ldots, n$, it holds
    \begin{equation}\label{3.4b}
 \lim_{t \rightarrow 0} \fint_{B_t(x)} \Delta v_k \mathop{\mathrm{d}\mathscr{H}^n}\leqslant \lim_{t \rightarrow 0} \fint_{B_t(x)} \left| \mathrm{Hess}\, v_k \right|^2 \mathop{\mathrm{d}\mathscr{H}^n}=0.
\end{equation}
\end{enumerate}
\end{itemize}

Let $C>0$ be a constant which may vary from line to line. Assume each $v_i$ is $C$-Lipschitz on $B_{\varepsilon_x}(x)$. For further discussion, the following proposition is required.

 \begin{prop}[{\cite[Proposition 4.8]{BMS23}},{\cite[(4.29)]{H25}}]\label{prop3.3}
Under the same assumptions and with the same notation introduced above, the following estimates hold for the function $r=|V|: B_{\varepsilon_x}(x) \to [0, \infty)$:
\begin{equation}\label{eqn3.4d}
     \fint_{B_t(x)} \left| \langle \nabla v_i,  \nabla v_j \rangle- \delta_{ij} \right| \mathop{\mathrm{d}\mathscr{H}^n}+\fint_{B_t(x)} \left||\nabla r| - 1\right|\mathop{\mathrm{d}\mathscr{H}^n}+ \fint_{B_t(x)} |\Delta r^2 - 2n| \mathop{\mathrm{d}\mathscr{H}^n}\leqslant t\,\xi(2t),
    \end{equation}
    where $\xi$ is a positive function that satisfies $\lim_{t\rightarrow 0}\xi(t)=0$, and is given by the following expression:
\[
\xi:t\mapsto C\max_i\sup_{s<t} \fint_{B_s(z)} {|\mathrm{Hess}\mathop{v_i}|}^2\,\mathrm{d}\mathscr{H}^n.
\] 
\end{prop}



\begin{lem}\label{lem3.3}
Under the same assumptions and with the same notation introduced above we have $\lim_{y\rightarrow x}r/\mathsf{d}_x=1$.
\end{lem}
\begin{proof}
We first show $\liminf_{y\rightarrow x}r/\mathsf{d}_x=1$ by contradiction. Assume the existence of $\epsilon>0$ and a sequence of points $y_j\rightarrow x$ such that $r(y_j)\leqslant (1-\epsilon)\,\mathsf{d}_x(y_j)$. Set $s_j=\mathsf{d}(x,y_j)$. Because $x\in \mathcal{R}_n$, 
\[
(X_j,\mathsf{d}_j,\mathscr{H}^n,x):=\left(X,{s_j}^{-1}\mathsf{d},{s_j}^{-n}\mathscr{H}^n_\mathsf{d},x\right)\xrightarrow{\mathrm{pmGH}}\left(\mathbb{R}^n,\mathsf{d}_{\mathbb{R}^n},\mathscr{L}^n,0_n\right).
\] 
Let $w_j^i:={s_j}^{-1}v_i$ be functions defined on $X_j$ ($i=1,\ldots,n$, $j\in\mathbb{N}$). By Proposition \ref{prop:blow.up.bx}, (\ref{eqn3.2a}) and (\ref{3.4b}), as $j\to \infty$, $\{w^i_j\}$ $H^{1,2}_{\mathrm{loc}}$-strongly converges to a linear function $w^i$ on $\mathbb{R}^n$ with
\[\langle\nabla_{\mathbb{R}^n} w^a,\nabla_{\mathbb{R}^n} w^b\rangle \equiv \delta_{ab},\ a,b=1,\ldots,n.
\]
In particular, since $v_i$ is $C$-Lipschitz, as $y_j\to y_\infty\in \partial B_1(0_n)$, Theorem \ref{AAthm} guarantees $w_j^i(y_j)\to w^i(y_\infty)$. As a result, we deduce a contradiction by the following observation.
\[
\left(\frac{r(y_j)}{\mathsf{d}(x,y_j)}\right)^2=\sum_{i=1}^n \frac{{v_i}^2(y_j)}{{s_j}^2}=\sum_{i=1}^n (w_j^i(y_j))^2\rightarrow 1,\ \mathrm{as}\ j\to\infty.
\]
Similarly, it can be proved that $\limsup_{y\rightarrow x}r/\mathsf{d}_x=1$, which is sufficient to reach a conclusion.
\end{proof}
Let us make a convention for the remainder of this section: all $t>0$ under consideration are sufficiently small such that $\mathbf{V}(B_t(x))\subset B_{2t}(0_n)$. Moreover, when referring to $\mathbf{V}^{-1}$, we are viewing $\mathbf{V}$ as a map from $B_{2t}(x)$ to $B_{4t}(0_n)$. Note that for any bounded Borel function $\varphi$ on $B_{2t}(0_n)$, the composition $\varphi\circ\mathbf{V}$ is a well-defined bounded measurable function on $B_t(x)$.

\begin{lem}\label{lem3.5}
Under the same assumptions and with the same notation introduced above we have
\begin{equation}\label{eqn3.15ccc}
\mathscr{H}^n(B_t(x)\setminus \mathbf{V}^{-1}(B_t(0_n)))=o(t^{n+1}),\ \mathrm{as}\ t\to 0.
\end{equation}
\end{lem}
\begin{proof}

The proof closely follows the approach outlined in \cite[Lemma 3.6]{BMS23} and \cite{H25}. Let $\tilde{r}=\min\{r,\mathsf{d}_x\}$ and $\Omega_t:=B_t(x)\cap \mathbf{V}^{-1}(B_t(0_n))$. Clearly $\Omega_t=\{\tilde{r}\leqslant t\}$. Additionally the locality of minimal relaxed slopes implies that 
\begin{align}\label{4.16}
	|\nabla \tilde{r}|=\chi_{\{r\leqslant \mathsf{d}_x\}}|\nabla r|+\chi_{\{r\geqslant \mathsf{d}_x\}},\end{align} indicating $\tilde{r}$ is a $C$-Lipschitz function. According to Lemma \ref{lem3.3}, we may assume 
	\begin{align}\label{4.17}
	\mathsf{d}_x\leqslant 2r\leqslant 4\,\mathsf{d}_x \  \mathrm{on}\ B_{2t}(x).
	\end{align}

 We claim that 
 \begin{align}\label{4.gsjirhgsao}
 \mathscr{H}^n(\Omega_t)=\omega_n t^n+o(t^{n+1})\ \mathrm{as}\  t\to 0.
 \end{align}

Let us set
\[
I:=\int_{ \Omega_t }\frac{r{|\nabla(r-\tilde{r})|}^2}{\tilde{r}^{n+1}}\,\mathrm{d}\mathscr{H}^n=\int_{\Omega_t }\frac{r\left({|\nabla r|}^2+{|\nabla \tilde{r}|}^2\right)}{\tilde{r}^{n+1}}\,\mathrm{d}\mathscr{H}^n-\int_{\Omega_t }\frac{\langle\nabla \tilde{r},\nabla r^2\rangle}{\tilde{r}^{n+1}}\,\mathrm{d}\mathscr{H}^n:=I_1-I_2.
\]

By \cite{BMS23}, \cite[(4.38)]{H25}, (\ref{4.16}) and (\ref{4.17}), we have
\begin{align}\label{fhieuoawhfoaehoifeah}
I\leqslant 4\int_{B_{2t}(x)} \frac{r{|\nabla(r-\tilde{r})|}^2}{{\mathsf{d}_x}^n}\,\mathrm{d}\mathscr{H}^n\leqslant o(t),\ \mathrm{as}\ t\to 0.
\end{align}
On the one hand, Theorem \ref{GGthm} implies
\[
	I_2=\int_0^t \frac{1}{s^{n+1}}\int_{\Omega_s} \Delta r^2\,\mathrm{d}\mathscr{H}^n\mathrm{d}s=\int_0^t \frac{1}{s^{n+1}}\int_{\Omega_s} (\Delta r^2-2n)\,\mathrm{d}\mathscr{H}^n\mathrm{d}s+2n\int_0^t \frac{\mathscr{H}^n(\Omega_s)}{s^{n+1}}\,\mathrm{d}s.
	\]
	Notice that
	\begin{align}\label{0504a}
	\int_{\Omega_t}\left||\nabla \tilde{r}|-1\right|\,\integral\leqslant \int_{B_{2t}(x)} \left||\nabla r|-1\right|\,\integral\leqslant o(t^{n+1}),\ \mathrm{as}\ t\to 0.
	\end{align}
	which in conjunction with (\ref{eqn3.4d}) yields
\begin{equation}\label{hfioehajiofho}
	\begin{aligned}
		&\left|I_2-2n\int_0^t \frac{1}{s^{n+1}}\int_{\Omega_s}|\nabla \tilde{r}|\,\integral\mathrm{d}s\right|\\
		\leqslant\ &\int_0^t \frac{1}{s^{n+1}}\int_{B_{2s}(x)} |\Delta r^2-2n|\,\mathrm{d}\mathscr{H}^n+\int_0^t \frac{1}{s^{n+1}}\int_{\Omega_s}\left||\nabla \tilde{r}|-1\right|\,\integral \leqslant o(t),\ \mathrm{as}\ t\to 0.
	\end{aligned}	
\end{equation}
On the other hand, observe from (\ref{4.16}) and (\ref{4.17}) that
\begin{equation}\label{4.19}
	\begin{aligned}
	&\left|I_1-\int_{ \Omega_t }\frac{2r}{\tilde{r}^{n+1}}\,
	\mathrm{d}\mathscr{H}^n\right|\leqslant \int_{ \Omega_t } \frac{2\left|{|\nabla r|}^2+{|\nabla \tilde{r}|}^2-2\right|}{\tilde{r}^n}\,\mathrm{d}\mathscr{H}^n
	\leqslant\int_{B_{2t}(x)}\frac{\left|{|\nabla r|}^2-1\right|}{{2}^{n-1}{\mathsf{d}_x}^n}\,\mathrm{d}\mathscr{H}^n\leqslant o(t),\ \mathrm{as}\ t\to 0,
	\end{aligned}
\end{equation}
where the last equality follows from a combination of co-area formula, integration by parts and (\ref{eqn3.4d}), as shown below:
\begin{equation}\label{0429a}
\begin{aligned}
	&\int_{B_{2t}(x)}\frac{\left|{|\nabla r|}^2-1\right|}{{\mathsf{d}_x}^n}\,\mathrm{d}\mathscr{H}^n=\int_0^{2t} \frac{1}{s^n}\int_{\partial B_s(x)}\left|{|\nabla r|}^2-1\right|\,\mathrm{dPer}(B_s(x),\cdot)\\
	=\ &\frac{1}{{(2t)}^n}\int_{B_{2t}(x)}\left|{|\nabla r|}^2-1\right|\,\integral+\int_0^{2t}\frac{n}{s^{n+1}}\int_{B_{2t}(x)}\left|{|\nabla r|}^2-1\right|\,\integral\mathrm{d}s\leqslant o(t),\ \mathrm{as}\ t\to 0.
\end{aligned}
\end{equation}
Similarly, we obtain
\[
\small{\begin{aligned}
	\left|\int_{ \Omega_t }\frac{r}{\tilde{r}^{n+1}}\,
	\mathrm{d}\mathscr{H}^n-\int_{ \Omega_t }\frac{r|\nabla\tilde{r}|}{\tilde{r}^{n+1}}\,
	\mathrm{d}\mathscr{H}^n\right|
	\leqslant \int_{ \Omega_t }\frac{r\left|{|\nabla \tilde{r}|}^2-1\right|}{\tilde{r}^{n+1}}\,
	\mathrm{d}\mathscr{H}^n\leqslant\frac{1}{2^{n}}\int_{B_{2t}(x)}\frac{\left|{|\nabla r|}^2-1\right|}{{\mathsf{d}_x}^n}\,\mathrm{d}\mathscr{H}^n\leqslant o(t),\ \mathrm{as}\ t\to 0.
\end{aligned}}
\]

If we show
\begin{align}\label{4.18}
\int_{ \Omega_t }\frac{(r-\tilde{r})|\nabla \tilde{r}|}{\tilde{r}^{n+1}}\,\mathrm{d}\mathscr{H}^n\leqslant o(t),\ \mathrm{as}\ t\to 0,
\end{align}
then using (\ref{4.19}) and co-area formula we get
\[
I_1=\int_{ \Omega_t } \frac{2|\nabla \tilde{r}|}{\tilde{r}^{n+1}}\,\mathrm{d}\mathscr{H}^n+o(t)=\int_0^t\frac{2}{s^{n+1}}\left(\frac{d}{ds}\int_{\Omega_s}|\nabla \tilde{r}|\,\integral\right)\,\mathrm{d}s+o(t),\ \mathrm{as}\ t\to 0.
\]
This together with (\ref{0504a}) and (\ref{hfioehajiofho}) then proves (\ref{4.gsjirhgsao}).

To see (\ref{4.18}), it suffices to verify
\begin{align}\label{fehjaoifhjeaoifhef}
\int_{ \Omega_t }\frac{(r-\tilde{r})|\nabla \tilde{r}|}{\tilde{r}^{n+1}}\,\mathrm{d}\mathscr{H}^n\leqslant C\int_{ B_{2t}(x)}\frac{r-\tilde{r}}{{\mathsf{d}_x}^{n+1}}\,\mathrm{d}\mathscr{H}^n\leqslant o(t),\ \mathrm{as}\ t\to 0.
\end{align}
Applying (1,1)-Poincar\'e inequality and volume comparison theorem gives
\begin{equation}\label{feahiofheaiohfaioeo}
\begin{aligned}
	&\left|\fint_{ B_{2t}(x)}(r-\tilde{r})\,\mathrm{d}\mathscr{H}^n-\fint_{ B_{t}(x)}(r-\tilde{r})\,\mathrm{d}\mathscr{H}^n\right|=\fint_{ B_{t}(x)}\left|(r-\tilde{r})-\fint_{ B_{2t}(x)}(r-\tilde{r})\,\mathrm{d}\mathscr{H}^n\right|\mathrm{d}\mathscr{H}^n\\
	&\leqslant C(K,n)\fint_{ B_{2t}(x)}\left|(r-\tilde{r})-\fint_{ B_{2t}(x)}(r-\tilde{r})\,\mathrm{d}\mathscr{H}^n\right|\mathrm{d}\mathscr{H}^n\leqslant C(K,n) t \fint_{B_{4t}(x)} {|\nabla (r-\tilde{r})|}\,\mathrm{d}\mathscr{H}^n.
\end{aligned}
\end{equation}
Owing to (\ref{4.17}), (\ref{fhieuoawhfoaehoifeah}) and H\"{o}lder inequality,
\begin{equation}\label{fheaiofhoeiahfoihfoiaejh}
	\begin{aligned}
		&\left(\fint_{B_{t}(x)} {|\nabla (r-\tilde{r})|}\,\mathrm{d}\mathscr{H}^n\right)^2 \leqslant \fint_{B_{t}(x)} {|\nabla (r-\tilde{r})|}^2\,\mathrm{d}\mathscr{H}^n\\
		=\ &\fint_{B_{t}(x)}\frac{{\mathsf{d}_x}^n}{r} \frac{r{|\nabla (r-\tilde{r})|}^2}{{\mathsf{d}_x}^{n}}\,\mathrm{d}\mathscr{H}^n
		\leqslant t^{n-1}\fint_{ B_{t}(x)}\frac{r{|\nabla (r-\tilde{r})|}^2}{{\mathsf{d}_x}^{n}}\,\mathrm{d}\mathscr{H}^n\leqslant o(1),\ \mathrm{as}\ t\to 0.
	\end{aligned}
\end{equation}
Because of $\lim_{y\to x}(r(y)-\tilde{r}(y))=0$, we obtain from (\ref{feahiofheaiohfaioeo}) and (\ref{fheaiofhoeiahfoihfoiaejh}) that 
\[
\begin{aligned}
\fint_{B_{2t}(x)}(r-\tilde{r})\,\mathrm{d}\mathscr{H}^n\leqslant  \sum_{i=0}^{\infty}\left|\fint_{ B_{2^{1-i}t}(x)}(r-\tilde{r})\,\mathrm{d}\mathscr{H}^n-\fint_{ B_{2^{-i}t}(x)}(r-\tilde{r})\,\mathrm{d}\mathscr{H}^n\right|\leqslant o(t),\ \ \mathrm{as}\ t\to 0.
\end{aligned}
\]
In a way similar to (\ref{0429a}), we deduce (\ref{fehjaoifhjeaoifhef}). Therefore (\ref{4.18}) and (\ref{4.gsjirhgsao}) hold. In particular, we have
\[
\mathscr{H}^n(B_t(x)\cap \mathbf{V}^{-1}(B_t(0_n)))=\omega_n t^n+ o(t^{n+1}),\ \mathrm{as}\ t\to 0.
\]
We conclude by recalling from \cite{BMS23} that $\mathscr{H}^n(B_t(x))=\omega_n t^n+ o(t^{n+1}),\ \mathrm{as}\ t\to 0.$
\end{proof}


Prior to proving Lemma \ref{lem3.4}, the following proposition is essential.
\begin{prop}[\cite{G13,ABS19b,H21}]\label{prop3.6a}
Let $(X,\mathsf{d},\mathfrak{m})$ be an $\mathrm{RCD}(0,N)$ space. If there exist harmonic functions $\{\varphi_i\}_{i=1}^n$ on the ball $B_{2}(x)$ with $n=\mathrm{dim}_{\mathsf{d},\mathfrak{m}}(X)$ such that $\langle \nabla \varphi_i,\nabla \varphi_j\rangle=\delta_{ij}$ $\mathfrak{m}$-a.e. on $B_{2}(x)$, then the map $\Phi:=(\varphi_1,\ldots,\varphi_n):B_1(x)\rightarrow B_1(0_n)$ is an isometry.
\end{prop}

\begin{lem}\label{lem3.4}
Under the same assumptions and with the same notation introduced above, for any sufficiently small $t>0$ we have 
\[
\mathscr{L}^n\left(\mathbf{V}\left(B_t(x)\right)\right)=\omega_n t^n+ o(t^{n+1}), \ \text{as}\ t\to 0.
\]
\end{lem}
\begin{proof}

Define
\[
\Lambda_t:=\left\{y\in B_{10t}(x)\cap\mathcal{R}_n:\max_{i,j}\sup_{s\in (0,4t)}\fint_{B_s(y)}\left|\langle \nabla v_i,\nabla v_j\rangle-\delta_{ij}\right|\leqslant \sqrt{\xi(2t)}\right\}.
\]
Then it immediately follows from the standard maximal function argument and  (\ref{eqn3.4d}) that
\begin{equation}\label{eqn3.17b}
\mathscr{H}^n\left(B_{10t}(x)\setminus\Lambda_t\right)\leqslant  t^{n+1}\sqrt{\xi(2t)}.
\end{equation} 
Since each $v_i$ is $C$-Lipschitz, we know
\begin{equation}\label{eqn3.7}
	\lim_{t\rightarrow 0}\frac{1}{\omega_n t^{n+1}}\int_{B_{t}(x)} \left|\left|\det \left(\langle\nabla v_i,\nabla v_j \rangle\right)\right|^{\frac{1}{2}}-1\right|\mathop{\mathrm{d}\mathscr{H}^n}=0.
\end{equation}

We claim 
\begin{equation}\label{eqn3.17}
\lim_{t\rightarrow 0}\sup_{z,w \in \Lambda_t,\mathsf{d}(z,w)\leqslant t}\left|\frac{|\mathbf{V}(z)-\mathbf{V}(w)|}{\mathsf{d}(z,w)}-1\right|=0.
\end{equation}

Assume otherwise. There exists $\tau_0>0$, $t_i\downarrow 0$ and $\{z_i\},\{w_i\}\subset X$ such that $z_i,w_i\in \Lambda_{t_i}$, $\mathsf{d}(z_i,w_i)\leqslant t_i$ and
\[
\left|\frac{|\mathbf{V}(z_i)-\mathbf{V}(w_i)|}{\mathsf{d}(z_i,w_i)}-1\right|\geqslant \tau_0.
\] 
Let $r_i=\mathsf{d}(z_i,w_i)$, $\varphi_{i,j}:={r_i}^{-1}(v_j-v_j(w_i))$. After passing to a subsequence, assume \[
\left(Z_i,\mathsf{d}_i,\mathfrak{m}_i,z_i\right):=\left(X,{r_i}^{-1}\mathsf{d},{\mathscr{H}^n(B_{r_i}(z_i))}^{-1}\mathscr{H}^n,z_i\right)\xrightarrow{\mathrm{pmGH}}\left(Z_\infty,\mathsf{d}_\infty,\mathfrak{m}_Z,z_\infty\right)
\]
for some RCD$(0,N)$ space $\left(Z_\infty,\mathsf{d}_\infty,\mathfrak{m}_Z,z_\infty\right)$. Since $\max_{j}\|\Delta v_j\|_{L^\infty(B_{\varepsilon_x}(x))}\leqslant C$, by Theorem \ref{AH18}, each $\{\varphi_{i,j}\}$ $H^{1,2}_{\mathrm{loc}}$-strongly and uniformly converges to a harmonic function $\varphi_j$ on $Z_\infty$. In particular,
\[
\begin{aligned}
&\int_{B_{10}^\infty(z_\infty)}|\langle\nabla \varphi_j,\nabla \varphi_k\rangle-\delta_{jk}|\mathop{\mathrm{d}\mathfrak{m}_Z}=\ \lim_{i\to \infty}\int_{B_{10}^{X_i}(z_i)}|\langle\nabla_i \varphi_{i,j},\nabla_i \varphi_{i,k}\rangle-\delta_{jk}|\mathop{\mathrm{d}\mathfrak{m}_i}
\\
= &\ \lim_{i\to \infty}\fint_{B_{ 10 r_i}(z_i)}|\langle\nabla v_j,\nabla v_k\rangle-\delta_{jk}|\mathop{\mathrm{d}\mathscr{H}^n}\leqslant \lim_{i\to \infty}\sqrt{\xi(2t_i)}=0.
\end{aligned}
\]
Moreover, applying \cite[Theorem 2.6]{BPS21} and Proposition \ref{prop3.6a} implies that $\mathrm{dim}_{\mathsf{d}_Z,\mathfrak{m}_Z}(Z)= n$ and  $(\varphi_1,\ldots,\varphi_n):B_2(z_\infty)\to B_2(0_n)$ is an isometry. We deduce a contradiction since $\{w_i\}$ converges to some $w_\infty\in \partial B_1(z_\infty)$. Due to the continuity of $\mathbf{V}$, (\ref{eqn3.17}) also holds for $\tilde{\Lambda}_t:=\Lambda_t\cap B_t(x)$.

We now show
\begin{equation}\label{eqn3.18a}
\mathbf{V}_\sharp\left(\left(\mathrm{det}(\langle\nabla v_i,\nabla v_j\rangle)\right)^{\frac{1}{2}}\mathscr{H}^n \llcorner_{\tilde{\Lambda}_t}\right)=\mathscr{L}^n.
\end{equation} 

Since $\mathbf{V}$ is $C$-Lipschitz, there exists a non-negative Radon-Nikodym differential $\zeta_t$ such that $\zeta_t\llcorner_{\mathbf{V}(B_{2t}(x))}\leqslant C$ and $\mathscr{L}^n\llcorner_{\mathbf{V}(\tilde{\Lambda}_t)}=\zeta_t\mathbf{V}_\sharp \mathscr{H}^n\llcorner_{\tilde{\Lambda}_t}$.
 It suffices to verify \begin{equation}\label{3.19}
 \zeta_t(\mathbf{V}(y))=\left(\mathrm{det}(\langle\nabla v_i,\nabla v_j\rangle)(y)\right)^{-\frac{1}{2}},
 \end{equation}
  when $y\in \mathcal{R}_n$ is both a Lebesgue point of $\langle\nabla v_i,\nabla v_j\rangle$ ($i,j=1,\ldots,n$) and a Lebesgue point of $\tilde{\Lambda}_t$.

Since $(\langle \nabla v_i,\nabla v_j\rangle(y))$ is positive definite, one can find an invertible matrix $E=(e_{ij})$ such that $E(\langle \nabla v_i,\nabla v_j\rangle(y))E^T=I_n$. Define $\mathbf{W}:=(w_1,\ldots,w_n)$, where $w_i:=\sum_j e_{ij}v_j$. Then clearly
\[
\lim_{s\to 0}\fint_{B_s(y)}\langle \nabla w_i,\nabla w_j\rangle\mathop{\mathrm{d}\mathscr{H}^n}=\langle \nabla w_i,\nabla w_j\rangle(y)=\delta_{ij}.
\]By a standard blow up argument as in Proposition \ref{prop:blow.up.bx}, we obtain
\[
\begin{aligned}
1=\lim_{s\to 0} \frac{\mathscr{L}^n\left(\mathbf{W}(B_s(y))\right)}{\mathscr{H}^n(B_s(y))}=\lim_{s\to 0} \frac{\mathscr{L}^n\left(\mathbf{V}(B_s(y))\right)}{\mathscr{H}^n(B_s(y))}\,\mathrm{det}E=\lim_{s\to 0} \frac{\mathscr{L}^n\left(\mathbf{V}(B_s(y))\right)}{\mathscr{H}^n(B_s(y))}\left(\mathrm{det}(\langle\nabla v_i,\nabla v_j\rangle)(y)\right)^{-\frac{1}{2}}.
\end{aligned}
\]
Provided $y$ is a Lebesgue point of $\tilde{\Lambda}_t$, we see
\[
\begin{aligned}
 &\mathscr{L}^n(\mathbf{V}(B_s(y)))\geqslant\mathscr{L}^n(\mathbf{V}(\tilde{\Lambda}_t\cap B_s(y)))=\mathscr{L}^n(\mathbf{V}(B_s(y)))-\mathscr{L}^n(\mathbf{V}(B_s(y)\setminus\tilde{\Lambda}_t))\\
\geqslant\ & \mathscr{L}^n(\mathbf{V}(B_s(y)))-C\mathscr{H}^n(B_s(y)\setminus\tilde{\Lambda}_t)\geqslant  \mathscr{L}^n(\mathbf{V}(B_s(y)))-o(s^n),\ \text{as }s\to 0, 
\end{aligned}
\]
and similarly
\[
\begin{aligned}
&\mathscr{L}^n(\mathbf{V}(\tilde{\Lambda}_t\cap B_s(y)))=\int_{\mathbf{V}(
\tilde{\Lambda}_t\cap B_s(y))}\zeta_t \mathop{\mathrm{d}\mathbf{V}_\sharp\mathscr{H}^n}\\
=\ &\int_{\tilde{\Lambda}_t\cap B_s(y)}\,\zeta_t\circ\mathbf{V}\,\integral=\int_{B_s(y)}\,\zeta_t\circ\mathbf{V}\mathop{\mathrm{d}\mathscr{H}^n}+o(s^n), \ \text{as }s\to 0.
\end{aligned}
\]

Therefore, up to a $\mathscr{H}^n$-negligible set, we have proved (\ref{3.19}). Finally, the conclusion follows from  (\ref{eqn3.17b}), (\ref{eqn3.7}) and (\ref{eqn3.18a}).
\end{proof}
\begin{cor}\label{cor3.7}
Under the same assumptions and with the same notation introduced above, for any bounded measurable function $\varphi$ on $B_{\varepsilon_0}(0_n)$ it holds that 
\[
\begin{aligned}
\int_{B_t(0_n)} \varphi\mathop{\mathrm{d}\mathscr{L}^n}=\int_{\mathbf{V}^{-1}(B_t(0_n))} \varphi\circ \mathbf{V}\left|\det \left(\langle\nabla v_i,\nabla v_j \rangle\right)\right|^{\frac{1}{2}}\mathop{\mathrm{d}\mathscr{H}^n}
+o(t^{n+1}),\ \mathrm{as}\ t\to 0.
\end{aligned} 
\]
\end{cor}

\begin{proof}[Proof of Theorem \ref{mcthm}]
By substituting $\varphi:=\chi_{\{u_n(v_1,\ldots,v_n)\geqslant u_n(x)\}}$ into Corollary \ref{cor3.7}, and using Lemma \ref{lem3.5}, (\ref{eqn3.17b}) and (\ref{eqn3.18a}), we obtain
\begin{equation}\label{eqn3.23}
 \mathscr{H}^n(B_t(x)\setminus \{y\in X:f(y)\leqslant f(x)\})=\mathscr{L}^n(B_t(0_n)\setminus \{u_n\leqslant u_n(x)\})+o(t^{n+1}),\ \mathrm{as}\ t\to 0.
\end{equation} 
Arguing as in \cite[Proposition 4.22]{H25}, we confirm that (\ref{1.2}) holds for $\mathscr{H}^n$-a.e. $x\in E\subset B_{\varepsilon_0}(x_0)$. 

To conclude it suffices to consider the case that $X$ is compact. For any $\delta>0$ and any $x\in \mathcal{R}_n\cap \mathrm{Leb}({|\nabla f|}^2)\cap \mathrm{Leb}({(\Delta f)}^2)\cap \{|\nabla f|\neq 0\}$, we can choose $\varepsilon_x>0$ and a Borel set $E_x\subset B_{5\varepsilon_x}(x)$ such that $\mathscr{H}^n\left(B_{5\varepsilon_x}(x)\setminus E_x\right)\leqslant C(n)\,\delta \mathscr{H}^n(B_{5\varepsilon_x}(x))$ and that (\ref{1.2}) holds for every $y\in E_x$. Consider the open covering
\[
\{B_{\varepsilon_x}(x)\}_{x\in \mathcal{R}_n\cap \mathrm{Leb}({|\nabla f|}^2)\cap \mathrm{Leb}({(\Delta f)}^2)\cap \{|\nabla f|\neq 0\}}.
\]
Vitali's covering theorem then verifies the existence of a countable collection of disjoint open balls $\{B_{\varepsilon_i}(x_i)\}$ such that $\mathcal{R}_n\cap \mathrm{Leb}({|\nabla f|}^2)\cap \{|\nabla f|\neq 0\}\subset \cup_i B_{5\varepsilon_i}(x_i)$. Moreover, by Theorem \ref{BGineq}, 
\[
\begin{aligned}
&\mathscr{H}^n\left(\{|\nabla f|\neq 0\}\cap \left(\bigcup_i (B_{5 \varepsilon_i}(x_i)\setminus E_{x_i})\right)\right)\leqslant C(n)\,\delta \,\sum_{i}\mathscr{H}^n(B_{5\varepsilon_i}(x_i))\\
&\leqslant C(n)\,\delta\, \sum_{i}\mathscr{H}^n(B_{\varepsilon_i}(x_i))\leqslant C(n)\,\delta\, \mathscr{H}^n(X) \to 0,\ \ \ \mathrm{as}\ \delta\to 0.
\end{aligned}
\]
This is sufficient to reach our conclusion.
\end{proof}
%


\bibliographystyle{alpha}
\bibliography{reff}

\bigskip

\end{document}